\documentclass[a4paper,12pt]{article}

\makeatletter
\@addtoreset{footnote}{page}
\makeatother

\usepackage{enumitem}

\usepackage{amsthm}
\usepackage{amsmath,amssymb,latexsym,amsfonts,mathrsfs}

\newcommand{\n}{\nonumber}

\renewcommand{\Re}{\mathop{\rm Re}}

\newcommand{\eps}{\ensuremath{\varepsilon}}

\renewcommand{\tilde}{\widetilde}
\renewcommand{\bar}{\overline}

\newcommand{\sbra}[1]{\!\left[ #1 \right]} 

\newcommand{\bC}{\ensuremath{\mathbb{C}}}
\newcommand{\bD}{\ensuremath{\mathbb{D}}}
\newcommand{\bE}{\ensuremath{\mathbb{E}}}

\newcommand{\bN}{\ensuremath{\mathbb{N}}}

\newcommand{\bP}{\ensuremath{\mathbb{P}}}
\newcommand{\bQ}{\ensuremath{\mathbb{Q}}}
\newcommand{\bR}{\ensuremath{\mathbb{R}}}

\newcommand{\cB}{\ensuremath{\mathcal{B}}}

\newcommand{\cD}{\ensuremath{\mathcal{D}}}

\newcommand{\cF}{\ensuremath{\mathcal{F}}}
\newcommand{\cG}{\ensuremath{\mathcal{G}}}

\newcommand{\cN}{\ensuremath{\mathcal{N}}}

\newcommand{\cQ}{\ensuremath{\mathcal{Q}}}
\newcommand{\cR}{\ensuremath{\mathcal{R}}}

\newcommand{\cX}{\ensuremath{\mathcal{X}}}
\newcommand{\cY}{\ensuremath{\mathcal{Y}}}

\theoremstyle{plain}
\newtheorem{Thm}{Theorem}[section]

\newtheorem{Lem}[Thm]{Lemma}
\newtheorem{Prop}[Thm]{Proposition}

\theoremstyle{definition}

\newtheorem{Def}[Thm]{Definition}
\newtheorem{Rem}[Thm]{Remark}

\numberwithin{equation}{section}

\makeatletter
\renewcommand\section{\@startsection {section}{1}{\z@}%
                                   {-3.5ex \@plus -1ex \@minus -.2ex}%
                                   {2.3ex \@plus.2ex}%
                                   {\normalfont\large\bf}}
\makeatother

\makeatletter
\renewcommand\subsection{\@startsection {subsection}{1}{\z@}%
                                   {-3.5ex \@plus -1ex \@minus -.2ex}%
                                   {2.3ex \@plus.2ex}%
                                   {\normalfont\normalsize\bf}}
\makeatother

\usepackage[sort]{natbib}
\setcitestyle{authoryear,comma,open={(},close={)}} 

\usepackage{hyperref}
\usepackage{mathtools}
\mathtoolsset{showonlyrefs}

\begin{document}
\begin{center}
	{\Large \bf 
		{Entrance boundary for standard processes with no negative jumps and its application to exponential convergence to the Yaglom limit}
	}
\end{center}
\begin{center}
	Kosuke Yamato (The University of Osaka)
\end{center}
\begin{center}
	{\small \today}
\end{center}

\begin{abstract}
	We study standard processes with no negative jumps under the entrance boundary condition.
	Similarly to one-dimensional diffusions, we show that the process can be made into a Feller process by attaching the boundary point to the state space.
	We investigate the spectrum of the infinitesimal generator in detail via the scale function, characterizing it as the zeros of an entire function. As an application, we prove that under the strong Feller property, the convergence to the Yaglom limit of the process killed on hitting the boundary is exponentially fast.
\end{abstract}


\section{Introduction}


For one-dimensional diffusions, \cite{FellerTheParabolicDifferentialEquations} introduced  a classification of the boundaries.
It is classically known that the classification gives clear description for various properties of diffusions (see e.g., \citet[Chapter 33]{Kallenberg-third} and \citet[Chapter V.7]{RogersWilliams2}).
In particular, when a boundary is classified as entrance, we can attach the boundary point to the state space and extend the process to have the Feller property (see e.g., \citet[Section 15]{FellerTheParabolicDifferentialEquations}, \citet[Problem 3.6.3]{Ito-McKean} and \citet[Theorem 33.13]{Kallenberg-third}).
Throughout this paper, the term ``Feller property'' refers to the property that the transition semigroup maps the set of continuous functions vanishing at infinity to itself.
In the present paper, we study standard processes with no negative jumps on $I = [0,\ell)$ or $(0,\ell) \ (\ell \in (0,\infty])$ satisfying a similar boundary condition at $\ell$.
Standard processes with no negative include various important processes such as one-dimensional diffusions, spectrally positive L\'evy processes, continuous-state branching processes, and their variants.

A successful generalization of the entrance boundary for processes with no negative jumps has been given in \cite{entranceBdry}.
They introduced the \textit{instantaneous entrance boundary condition} for $\ell$ by a limit condition of hitting times (see \eqref{instantaneousEntrance} for definition), which is equivalent to Feller's entrance condition for one-dimensional diffusions.
Under the condition and a certain regularity assumption on the transition semigroup, they showed that the process can be extended to a Feller process on $I \cup \{\ell\}$.
We generalize this result to the processes with killing and relax the assumptions on the transition semigroup to the corresponding ones on the resolvent.

We also establish a connection between the boundary condition and the theory of scale functions.
Feller's original definition of the entrance boundary for one-dimensional diffusions was given by the integrability of the scale function with respect to the speed measure.
In \cite{NobaGeneralizedScaleFunc}, a scale function for standard processes with no negative jumps was introduced under mild technical assumptions.
He showed that the exit times and potential densities on intervals are represented by it as in the case of one-dimensional diffusions, which we will recall in Section \ref{section:scaleFunc}.
In \cite{noba2023analytic}, the entrance and non-entrance boundary classification of $\ell$ for such processes was introduced in relation to the study of quasi-stationary distributions.
The classification is based on whether the scale function is integrable with respect to the reference measure, which is a direct extension of Feller's original definition.
As we will check in Proposition \ref{prop:twoEntranceConditions}, this entrance condition is equivalent to the instantaneous entrance boundary condition in \cite{entranceBdry} under a mild assumption.

Once we extend the process to have the Feller property, we have its infinitesimal generator by the Hille-Yosida theorem.
One of the main purposes of the present paper is to study its spectrum when the process is killed on hitting boundaries.
The scale function is crucial for this purpose.
Under the entrance condition, the resolvent density killed on hitting zero is represented by a meromorphic function on the complex plane, which is a rational function of the scale function.
This implies that the spectrum of the infinitesimal generator coincides with the zeros of the entire function given as the denominator of the resolvent density.
This is reminiscent of the classical Fredholm integral equation theory (see e.g., \cite{MR65391,MR2772421,MR1892228}) and shows that the Fredholm determinant is given by the scale function.
We also show that the compactness of the resolvent and determine the spectrum of the infinitesimal generator on $L^{p}$-spaces $(p \in [1,\infty))$ when a quasi-stationary distribution exists.

The other main purpose is to show the exponential convergence of the \textit{Yaglom limit} for the standard process $(X_{t},\bP_{x})$ on $[0,\ell]$ with no negative jumps killed on hitting zero, under the assumption of the strong Feller property.
We say that a distribution $\nu$ on $(0,\ell]$ is the Yaglom limit of $X$ when the following holds in a certain sense of convergence:
\begin{align}
	\bP_{x}[X_{t} \in dy \mid \tau_{0} > t] \xrightarrow[t \to \infty]{} \nu(dy) \quad \text{for every $x \in (0,\ell]$},
\end{align}
where $\tau_{0} := \inf \{ t \geq 0 \mid X_{t} = 0 \}$.
The existence of the Yaglom limit in the total variation distance was shown in \citet[Theorem 5.24]{noba2023analytic} under a certain continuity of the transition semigroup (see Theorem \ref{thm:YaglomLimit}) as an application of the $R$-theory for general Markov processes developed in \cite{TuominenTweedie}.
Imposing a stronger assumption, the strong Feller property (SF) presented in Section \ref{section:exponentialConvergence}, we show that the convergence is exponentially fast.
The proof is essentially due to the results in \cite{SchillingWang} and \cite{KolbSavov}.
In \cite{SchillingWang}, they showed compactness of the transition semigroup for a general class of Markov processes on $\bR^{d}$ under the strong Feller property.
\cite{KolbSavov} applied their argument to L\'evy processes killed at the exit time of a finite interval and showed the exponential convergence to the Yaglom limit.
By following these studies and combining them with the extension of the state space, we show the compactness of the transition semigroup and derive the exponential convergence.
We also observe the exponential ergodicity of the conditional limit process to avoid zero.
In appendix, we treat the exponential convergence to the Yaglom limit for processes killed on hitting either of the boundaries.

We briefly review previous studies on the exponential rate of convergence to the Yaglom limit.
In \cite{ChampagnatVillemonaisPTRF}, the authors provided a necessary and sufficient condition for exponential convergence in the total variation distance, based on a kind of Dobrushin condition and Harnack inequality.
They also established in \cite{ChampagnatVillemonais} a Lyapunov-type criterion for exponential convergence in weighted total variation distances.
While their conditions are quite general and applicable to a wide range of examples, they are abstract and can be difficult to verify in general.
In particular, the application to Markov processes with jumps in continuous state spaces seems limited.
It is worth noting, however, that verifying our assumption, the strong Feller property, can also be challenging.
Our approach differs from theirs in that we utilize the spectral theory.
The advantage lies in explicitly providing a lower bound of the convergence rate as a spectral gap, whereas their method provides it in an abstract manner, making it quantitatively intractable.
In \cite{KolbSavov}, they studied L\'evy processes killed on exiting a finite interval.
Their method relies mainly on abstract machinery, specifically semigroup theory and spectral theory for the space of continuous functions.
A crucial role is played by a result on the compactness of the transition semigroup given in \cite{SchillingWang}.
While they suggest that their method is applicable beyond L\'evy processes (see \citet[Section 3]{KolbSavov}), the compactness of the state space is an essential limitation.
Thus, extending their arguments to non-compact spaces is not straightforward.
Our result on the Yaglom limit shows that their method can still work for non-compact state spaces under the entrance condition.

Finally, we note that the results of the present paper are also valid with trivial modifications for discrete state space analogs, namely, continuous-time Markov chains on $\bN$ whose size of downward jumps is at most one.
Such processes are called \textit{downward skip-free Markov chains}.
The scale function and quasi-stationary distributions for such processes were studied in \cite{QSD_downward_skip-free}, where the counterparts of the results used in this paper for discrete state spaces are provided.

\subsection*{Outline of the paper}

In Section \ref{section:preliminary}, we establish the necessary definitions and fix the notation that will be used throughout the paper.
We also recall the definition and basic properties of the scale function, as well as review previous results on quasi-stationary distributions and the Yaglom limit.
In Section \ref{section:entranceBoundaryAndFellerProperty}, we study the extension of the state space under the entrance boundary condition.
In Section \ref{section:spectralTheory}, we study the spectrum of the infinitesimal generator and characterize it as the zeros of an entire function.
In Section \ref{section:exponentialConvergence}, we show the exponential convergence to the Yaglom limit under the strong Feller property.
In Appendix \ref{appendix:accessible}, we study the exponential convergence to the Yaglom limit when the right boundary $\ell$ is accessible.

\subsection*{Acknowledgement}

The author was supported by JSPS KAKENHI grant no.\ 24K22834 and JSPS Open Partnership Joint Research Projects grant no.\ JPJSBP120209921.
This research was supported by RIMS and by ISM.

	

\section{Preliminary} \label{section:preliminary}

In the following, we fix a one-dimensional standard process $X=(\Omega, \cF, \cF_t, X_t, \theta_t, \bP_x)$ whose state space is an interval $I \subset [-\infty,\infty]$ with the cemetery $\partial$.
For the definition of standard processes, see e.g., \cite[(9.2) in Chapter 1]{BluGet1968}.
We write $I_{\Delta} := I \cup \{\partial \}$, $\ell_{1} := \inf I$ and $\ell_{2} := \sup I$.
Denote the lifetime by $\zeta$.
We follow the convention that $f(\partial) = 0$ for every measurable function $f$ on $I_{\Delta}$.
We denote the transition operator by $p_{t}$, that is, $p_{t}f(x) := \bE_{x}[f(X_{t})]$ for every bounded measurable function $f$.
For $x \in I$, we write
\begin{align}
I_{\geq x}:= I \cap [x, \infty], \quad 
I_{>x}:= I \cap (x, \infty], \quad 
I_{\leq x }:=I \cap [-\infty, x], \quad
I_{<x }:=I \cap [-\infty, x). \n
\end{align}
As an essential assumption for the processes, we suppose the absence of negative jumps throughout the paper:
\begin{align}
	\bP_{z}[ \tau_{x}^{-} < \tau_{y}] = 0 \quad (\text{$x, y, z \in I$ with $x < y < z$}), \label{absenceOfNegativeJumps1}
\end{align}
where $\tau_{A} := \inf \{ t > 0 \mid X_{t} \in A \}$ denotes the first hitting time of the set $A$ and we especially write 
\[
\tau_{x} := \tau_{\{x\}}, \quad \tau_{x}^{-} := \tau_{I_{\leq x}} \quad\text{and} \quad \tau_{x}^{+} := \tau_{I_{\geq x}} \ (x \in I).
\]
We note the following condition equivalent to \eqref{absenceOfNegativeJumps1}, which may more clearly represent the absence of negative jumps:

\begin{Prop} \label{prop:absenceOfNegativeJumpsEquivalence}
	Condition \eqref{absenceOfNegativeJumps1} holds if and only if
	\begin{align}
		\bP_{x}[ X_{t} - X_{t-} < 0 \quad \text{for some $t > 0$}] = 0 \quad \text{for every $x \in I$}. \label{absenceOfNegativeJumps2}
	\end{align} 	
\end{Prop}
We prove this proposition in Appendix \ref{appendix:absenceOfNegativeJumpsEquivalence}.

\subsection{The scale function}

The scale function is the primary tool used throughout Section \ref{section:spectralTheory} and beyond.
We prepare necessary notions for the scale function and recall the definition and basic properties.

\subsubsection{Local times, potential densities, and excursion measures} \label{section:localTimeExcursionMeasure}

To define the scale function, we need an excursion measure away from a point.
Here we set assumptions to ensure the existence and recall basic facts.
See \cite[Section 2.1]{noba2023analytic} for details.

For $q > 0$, we denote the $q$-resolvent of $X$ by $R^{(q)}$:
for every non-negative measurable function $f$
\begin{align}
	R^{(q)}f(x) := \bE_{x}\left[ \int_{0}^{\infty}\mathrm{e}^{-qt}f(X_{t})dt \right] =  \int_{0}^{\infty}\mathrm{e}^{-qt}p_{t}f(x)dt.\label{probabilisticResolvent}
\end{align}
We introduce the following technical conditions:
\begin{description}
	\item[(A1)] The map $(x,y) \longmapsto \bE_{x}[\mathrm{e}^{-\tau_{y}}]$ is $\cB(I) \otimes \cB(I)$-measurable, where $\cB(I)$ is the Borel sets of $I$. 
	\item[(A2)] For every $x,y \in I$ with $x < y$, we have $\bP_{y}[\tau_{x} < \infty] > 0$.
	\item[(A3)] There is \textit{a reference measure} $m$ of $X$,
	that is, there exists a Radon measure $m$ on the set of universally measurable sets $\cB(I)^{\ast}$ of $I$, and we have for any $A \in \cB(I)^{\ast}$ that
	\begin{align}
		R^{(1)}1_{A}(x) = 0 \ \text{for every $x \in I$ if and only if } m(A) = 0. \label{}
	\end{align}
	For details, see, e.g., \citet[p.196]{BluGet1968}.
\end{description}
In the rest of Section \ref{section:preliminary}, we always assume (A1)-(A3).
Condition (A2) ensures that the reference measure has a full support.

\begin{Prop}[{\citet[Proposition 2.1]{noba2023analytic}}]\label{Prop101}
	The measure $m$ satisfies $m(x, y)>0$ for $x, y \in I$ with $x< y$. 
\end{Prop}

Under conditions (A1)-(A3), by virtue of the general result \cite[Theorem 18.4]{GemanHorowitz}, there exists a family of processes $(L^{x})_{x \in I}$ with $L^{x} = (L^{x}_{t})_{t \geq 0}$ which we call the \textit{local times} of $X$ satisfying the occupation time formula:
for every $t\geq 0$, non-negative measurable function $f$ and $x \in I$,
\begin{align}
	\int_{0}^{t}f(X_{s})ds &= \int_{I}f(y)L^{y}_{t}m(dy) \quad \bP_x\text{-a.s}. \label{occupationTimeFormula} 
\end{align}
Note that this formula implies
\begin{align}
	\bE_x \sbra{\int_{0}^{T} e^{-qt} f(X_t) dt} = \int_{{I}}f(y)\bE_{x}\left[ \int_{[0, T)}\mathrm{e}^{-qt} dL^{y}_{t} \right] m(dy) \label{potentialDensityFormula}
\end{align}
for $q > 0$ and non-negative random variable $T$.
Thus, for each $x \in I$ the function $y \mapsto \bE_{x}\left[ \int_{[0,T)}\mathrm{e}^{-qt} dL^{y}_{t} \right]$ is a $q$-potential density of $X$ starting from $x$ and killed at $T \wedge \zeta$.
Through the local time, the existence of an excursion measure away from a point follows (see e.g., \citet[Section IV]{BertoinLevy} and \cite{Itoexcursion}).
Specifically, for every $x \in I$, there exists an excursion measure $n_{x}$ on $\bD_{x}$, the space of c\`adl\`ag paths from $[0,\infty)$ to $I_{\Delta}$ which starts from and stops at $x$, such that
\begin{align}
	-\log \bE_{x}[\mathrm{e}^{-q\eta^{x}_{1}}] =\delta_{x} q + n_{x}[1 - \mathrm{e}^{-q\tau_{x}}]\quad (q\geq 0)  \label{107}
\end{align}
for some $\delta_{x}\geq 0$, where $\eta^{x}$ is the inverse local time of $x$, the right-continuous inverse of $t \mapsto L^{x}_{t}$. 

\subsubsection{Definition of the scale function and basic formulas} \label{section:scaleFunc}

The scale function for standard processes with no negative jumps was introduced in \cite{NobaGeneralizedScaleFunc}.
Set $\tau_{\ell_{1}}^{-} (= \tau_{\ell_{1}}) := \lim_{x \to \ell_{1}}\tau_{x}^{-}$ when $\ell_{1} \not\in I$ and define
\begin{align}
	\tilde{I} := \left\{
		\begin{aligned}
			&I \cup \{\ell_{1}\} &(\ell_{1} \not\in I \text{ and }\bP_{x}[\tau_{\ell_{1}} < \infty] > 0 \text{ for every (or some) } x \in I) \\
			&I &(\text{otherwise})
		\end{aligned}
	\right.
	.
\end{align}
The following is the definition of the scale function, which plays a fundamental role throughout the paper.

\begin{Def}[{\citet[Definition 3.1]{NobaGeneralizedScaleFunc}}] \label{def:scaleFunc}
	For $q \geq 0$, define the \textit{$q$-scale function} $W^{(q)}: \tilde{I} \times I \to [0,\infty)$ by
	\begin{align}
		W^{(q)}(x,y) := \frac{1\{x \leq y\}}{n_{y}[\mathrm{e}^{-q\tau_{x}^{-}}, \tau_{x}^{-} < \infty]}.			
		\label{scaleFunc}
	\end{align}
	where we consider $1/ \infty = 0$.
	We especially write $W := W^{(0)}$.
\end{Def}

We recall some basic properties of the scale function.
See \cite{NobaGeneralizedScaleFunc} and \cite{noba2023analytic} for the proof.
Actually, the results presented below are slightly generalized to include the case in which $\ell_{1} \not\in I$ and $\bP_{x}[\tau_{\ell_{1}} < \infty] > 0$.
We can, however, easily extend to the case by an obvious modification of the proof or limit procedure for the original results.
Thus, we omit the proof.

\begin{Prop}[{\citet[Proposition 2.4, 2.6, 3.2]{noba2023analytic}}] \label{prop:propertiesOfScaleFunc}
	For $q \geq 0$, the following hold:
	\begin{enumerate}
		\item For $x,y \in \tilde{I}$ with $x < y$, we have $W^{(q)}(x,y) \in (0,\infty)$.
		\item For $y \in I_{>\ell_{1}}$, the function $\tilde{I}_{<y} \ni x \mapsto W^{(q)}(x,y)$ is continuous.
		\item For $x,y,z \in \tilde{I}$ with $x < y < z$, we have $\bE_{y}[\tau_{x}^{-} \wedge \tau_{z}^{+} \wedge \zeta] < \infty$ and
		\begin{align}
			\int_{(x, z]}W^{(q)}(x,u) m(du) < \infty. \label{eq24}
		\end{align}
		In addition, we have
		\begin{align}
			\lim_{x \to z} \int_{(x,z]}W^{(q)}(x,u) m(du) = 0. \label{eq17}
		\end{align}
	\end{enumerate}
\end{Prop}

The following two formulas, the exit time formula (Proposition \ref{prop:exitProblem1} and \ref{prop:exitProblemZ}) and the potential density formula (Theorem \ref{thm:potentialDensity}), are fundamental to the theory of the scale function.
We interpret every integration on $(x,y)$ with $y \leq x$ as zero throughout the paper.

\begin{Prop}[{\citet[Theorem 3.4]{NobaGeneralizedScaleFunc}}] \label{prop:exitProblem1}
	Let $q \geq 0$.
	For $x \in \tilde{I}$ and $y,z \in I$ with $x {\leq} y {\leq} z$ and $x \neq z$,
	we have
	\begin{align}
		\bE_{y}[\mathrm{e}^{-q \tau^-_{x}}, \tau^-_{x} < \tau_{z}^{+}]
		= \frac{W^{(q)}(y,z)}{W^{(q)}(x,z)}. \label{eq56}
	\end{align}
\end{Prop}

\begin{Prop}[{\citet[Corollary 3.7]{NobaGeneralizedScaleFunc}}] \label{prop:exitProblemZ}
	For $q \geq 0$, $x \in \tilde{I}$ and $y \in I$, define
	\begin{align}
		Z^{(q)}(x,y) := 
			&1 + q\int_{(x,y)}W^{(q)}(x,u)m(du).
		\label{scaleFuncZ}
	\end{align}
	Then we have for $x \in \tilde{I}$ and $y,z \in I$ with $x \leq y \leq z$ and $x<z$
	\begin{align}
		\bE_{y}[\mathrm{e}^{-q(\tau_{z}^{+} \wedge \zeta)}, (\tau_{z}^{+} \wedge \zeta) < \tau^-_{x}] = Z^{(q)}(y,z) - \frac{W^{(q)}(y,z)}{W^{(q)}(x,z)} Z^{(q)}(x,z). \label{eq54}
	\end{align}
\end{Prop}

\begin{Rem} \label{rem:W(x,x)m(x)=0}
	From \cite[Remark 2.5]{noba2023analytic}, we always have
	\begin{align}
		W^{(q)}(x,x)m\{x\} = 0 \quad (q \geq 0,\
		 x \in I \setminus \{ \ell_{1}\}), \label{}
	\end{align}
	which implies that the function $I \ni x \mapsto Z^{(q)}(x,y)$ is continuous.
\end{Rem}

\begin{Thm}[{\citet[Theorem 3.6]{NobaGeneralizedScaleFunc}}] \label{thm:potentialDensity}
	Let $q \geq 0$ and $x \in \tilde{I}$ and $y,z,u \in I$ with $u \in (x,z)$ and $y \in [x,z]$.
	We have
	\begin{align}
		\bE_{y}\left[ \int_{0}^{\tau^-_{x} \wedge \tau_{z}^{+}} \mathrm{e}^{-qt} dL^{u}_{t} \right]
		= &\frac{W^{(q)}(x,u)W^{(q)}(y,z)}{W^{(q)}(x,z)} - W^{(q)}(y,u) \label{} \\
		= &W^{(q)}(x,u) \bE_{y}[\mathrm{e}^{-q\tau^-_{x}}, \tau^-_{x} < \tau_{z}^{+}] - W^{(q)}(y,u). \label{potentialDensity}
	\end{align}
	Consequently, we have from \eqref{potentialDensityFormula}
	\begin{align}
		\int_{0}^{\infty}\mathrm{e}^{-qt}\bP_{y}[X_{t} \in du, \tau_{x}^{-}\wedge \tau_{z}^{+} \wedge \zeta > t]dt = (W^{(q)}(x,u) \bE_{y}[\mathrm{e}^{-q\tau^-_{x}}, \tau^-_{x} < \tau_{z}^{+}] - W^{(q)}(y,u))m(du).
	\end{align}
\end{Thm}

\subsubsection{Analyticity of the scale function}

In Definition \ref{def:scaleFunc}, the scale function $W^{(q)}$ was defined for $q \geq 0$.
It was proved in \cite{noba2023analytic} that the function $[0,\infty) \ni q \mapsto W^{(q)}(x,y)$ can be analytically extended to the entire function for each $x \in \tilde{I}$ and $y \in I$.
For measurable functions $f,g: \tilde{I} \times I \to \bR$ and $x \in \tilde{I}$ and $y \in I$, we denote
\begin{align}
	f \otimes g (x,y) := \int_{(x,y)}f(x,u)g(u,y) m(du), \label{}
\end{align}
when $f$ and $g$ are non-negative or
\begin{align}
	\int_{(x,y)}|f(x,u)g(u,y)| m(du) < \infty. \label{}
\end{align}
We write $f^{\otimes 1} := f$ and $f^{\otimes n} := f \otimes f^{n-1} \ (n \geq 2)$.
The following result characterizes the scale function as the solution of a Volterra integral equation.

\begin{Thm}[{\citet[Proposition 3.3, Theorem 3.5]{noba2023analytic}}]
	For every $q \geq 0$ and $x \in \tilde{I}$ and $z \in I$ with $x < z$, the function $f = W^{(q)}(x,\cdot)$ is a solution of the following equation:	\begin{align}
		f(y) = W(x,y) + q \int_{(x,y)}f(u) W (u,y) m(du) \quad (y \in I \cap [x,z]). \label{IntegralEqW-c} 
	\end{align}
	In addition, this solution is unique in the following sense: 
	for an $m$-null set $N \subset I$, if $f:I \cap [x,z] \to \bR$ satisfies \eqref{IntegralEqW-c} and
	\begin{align}
		\int_{(x, y)}|f(u)|W(u,y) m(du) < \infty  \label{eq04-c}
	\end{align}
	for $y \in [x,z] \setminus N$,
	then it follows $f(y) = W^{(q)}(x,y)$ for $y \in [x,z] \setminus N$.
\end{Thm}

The equation \eqref{IntegralEqW-c} can be easily solved by iteration procedure, and we have a series expansion of the scale function:
\begin{align}
	W^{(q)}(x,y) = \sum_{n \geq 0}q^{n}W^{\otimes (n+1)}(x,y). \label{eq43}
\end{align}
The convergence of the RHS for every $q \in \bC$ is ensured by the following proposition, and hence we can analytically extend the function $[0,\infty) \ni q \mapsto W^{(q)}(x,y)$ to an entire function.
We denote
\begin{align}
	\bar{W}(x,y) := \int_{(x,y)}W(x,u)m(du) \quad (x \in \tilde{I},\ y \in I \cup \{ \ell_{2} \}). \label{}
\end{align}

\begin{Prop}[{\citet[Proposition 3.3]{noba2023analytic}}] \label{prop:defOfM}
	For every $n \geq 1$ and $x \in \tilde{I}$ and $y \in I$ with $x < y$, we have
	\begin{align}
		W^{\otimes n} (x,y) \leq \frac{W(x,y)}{(n-1)!} \bar{W}(x,y)^{n-1}. \label{eq01}
	\end{align}
	This implies that for $q \in \bC$ 
	\begin{align}
		\sum_{n \geq 0}|q|^{n} W^{\otimes (n+1)}(x,y) \leq W(x,y) \mathrm{e}^{|q|\bar{W}(x,y)}. \label{eq03}
	\end{align}
\end{Prop}

We also have the following useful identity.

\begin{Prop}[{\citet[Proposition 3.6]{noba2023analytic}}] \label{prop:resolventIdentityW}
	For $q,r \in \bC$ and $x \in \tilde{I}$ and $y \in I$, we have
	\begin{align}
		W^{(q)}(x,y) - W^{(r)}(x,y) = (q-r) W^{(q)} \otimes W^{(r)} (x,y) = (q-r) W^{(r)} \otimes W^{(q)} (x,y). \label{resolventIdentityW}
	\end{align}
\end{Prop}

A similar characterization of $Z^{(q)}$ also holds,
and we can extend the function $[0,\infty) \ni q \mapsto Z^{(q)}(x,y)$ to an entire function.

\begin{Thm}[{\citet[Theorem 3.10]{noba2023analytic}}] \label{thm:Zexpansion}
	For every $q \geq 0$ and $x \in \tilde{I}$ and $z \in I$ with $x < z$, the function $f = Z^{(q)}(\cdot,z)$ is a solution of the following equation:
	\begin{align}
		f(y) = 1 + q \int_{(y,z)} W (y,u)f(u) m(du) \quad (y \in [x,z]). \label{IntegralEqZ-c} 
	\end{align}
	In addition, this solution is unique in the following sense: 
	for an $m$-null set $N \subset I$, if $f:[x,z] \to \bR$ satisfies \eqref{IntegralEqZ-c} and
	\begin{align}
		\int_{(y,z)}W(y,u)|f(u)| m(du) < \infty  \label{eq04Z-c}
	\end{align}
	for $y \in [x,z] \setminus N$,
	then it follows $f(y) = Z^{(q)}(y,z)$ for $y \in [x,z] \setminus N$.
	Consequently, the function $q \mapsto Z^{(q)}(y,z)$ is the following series expansion:
	\begin{align}
		Z^{(q)}(x,y) = 1 + q\bar{W}(x,y) + \sum_{n \geq 2}q^{n} W^{\otimes (n-1)} \otimes \bar{W}(x,y), \label{}
	\end{align}
	where the convergence of the RHS is ensured by Proposition \ref{prop:characterizationOfZ}.
\end{Thm}

\begin{Prop}[{\citet[Theorem 3.10]{noba2023analytic}}] \label{prop:characterizationOfZ}
	For every $n \geq 1$ and $x \in \tilde{I}$ and $y \in I$, we have
	\begin{align}
		W^{\otimes n} \otimes \bar{W}(x,y) = \int_{(x,y)}W^{\otimes(n+1)}(x,u)m(du) \leq \frac{\bar{W}(x,y)^{n+1}}{n!}. \label{} 
	\end{align}
	This implies that for $q \in \bC$
	\begin{align}
		\bar{W}(x,y) + \sum_{n \geq 1}|q|^{n} W^{\otimes n} \otimes \bar{W}(x,y) \leq \bar{W}(x,y)\mathrm{e}^{|q|\bar{W}(x,y)}. \label{}
	\end{align}
\end{Prop}

\begin{Prop} \label{prop:resolventIdentityZ}
	For $q,r \in \bC$ and $x \in \tilde{I}$ and $y \in I$, we have		\begin{align}
		Z^{(q)}(x,y) - Z^{(r)}(x,y) = (q-r) W^{(q)} \otimes Z^{(r)}(x,y)=(q-r) W^{(r)} \otimes Z^{(q)}(x,y). \label{resolventIdentityZ}
	\end{align}
\end{Prop}

\subsection{Quasi-stationary distributions and the Yaglom limit} \label{section:QSDprevious}

Let us recall the results in \cite{noba2023analytic} on the existence of a quasi-stationary distribution of $X$ killed on hitting the left end point of $I$.
For simplicity, we consider the case where $\ell_{1} = 0$, $\ell_{2} = \infty$, and $0 \not \in I$, that is, the state space is $I=(0,\infty)$ or $(0,\infty]$.
Since we can always reduce the state space to this case by an order-preserving homeomorphism, we do not lose any generality.
In this subsection, we assume that $X$ is killed on hitting zero and is not killed from inside; $\zeta = \tau_{0}$.
The main assumption in \cite{noba2023analytic} was the following:
\begin{description}
	\item[(Q)] $\bP_{x}[\tau_{y} < \infty] > 0$ and $\bP_{x}[\tau_{0} < \infty] = 1 \quad (x \in I, \ y \in (0,\infty))$.	 
\end{description}
We say that a distribution $\nu$ on $I$ is a quasi-stationary distribution when
\begin{align}
	\bP_{\nu}[X_{t} \in dy \mid \tau_{0} > t] = \nu(dy) \quad (t \geq 0). \label{}
\end{align}
In \cite{noba2023analytic}, the existence of a quasi-stationary distribution
was completely characterized.
Let us define \textit{the decay parameter}
\begin{align}
	\lambda_{0} := \sup \{ \lambda \geq 0 \mid \bE_{x}[\mathrm{e}^{\lambda \tau_{0}}] < \infty \quad \text{for every $x \in I$ } \}. \label{decayParameter}
\end{align}
Note that under (Q), we always have $\lambda_{0} < \infty$ (see e.g., \cite[Proposition 5.1]{noba2023analytic}).
The boundary classification of $\infty$ was introduced as follows:

\begin{Def} \label{def:boundaryClassification}
	We say that $\infty$ is an \textit{entrance boundary} if
	\begin{align}
		\bar{W}(b,\infty) < \infty \quad \text{for some $b \in (0,\infty)$}, \label{entranceCondition}
	\end{align}
	and a \textit{non-entrance boundary} otherwise.
\end{Def}

\begin{Rem}
	The present definition of the entrance boundary is apparently more general than in \cite{noba2023analytic}, where it was defined as $\bar{W}(0,\infty) < \infty$.
	These definitions are equivalent when $\zeta = \tau_{0}$ and $\bP_{x}[\tau_{0} < \infty] > 0 \ (x \in I)$, particularly under (Q).
	Indeed, when $\infty \in I$, we have $\bar{W}(0,\infty),\ \bar{W}(b,\infty) < \infty$ from Proposition \ref{prop:propertiesOfScaleFunc} (iii).
	When $\infty \not\in I$, by the quasi-left continuity and $\zeta = \tau_{0}$, it follows $\bP_{x}[\lim_{z \to \infty}\tau_{z}^{+} = \infty] = 1$ for every $x \in I$.
	This and Proposition \ref{prop:exitProblem1} imply $0 < \bP_{b}[\tau_{0} < \infty] = \lim_{z \to \infty}\bP_{b}[\tau_{0} < \tau_{z}^{+}] = \lim_{z \to \infty} W(b,z)/W(0,z)$.
	Thus, the integrability of functions $u \mapsto W(0,u),W(b,u)$ coincides.
\end{Rem}

The following characterizes the existence of quasi-stationary distributions.

\begin{Thm}[{\cite[Theorem 5.3]{noba2023analytic}}] \label{thm:existenceOfQSD}
	Suppose (Q).
	Let $\cQ$ be the set of quasi-stationary distributions.
	The following holds:
	\begin{enumerate}
		\item If $\infty$ is entrance, then $\lambda_{0} > 0$ and $\cQ = \{ \nu_{\lambda_{0}} \}$.
		\item If $\infty$ is non-entrance and $\lambda_{0} > 0$, then $\cQ = \{ \nu_{\lambda} \}_{\lambda \in (0,\lambda_{0}]}$.
	\end{enumerate}
	Here 
	\begin{align}
		\nu_{\lambda}(dx) := \lambda W^{(-\lambda)}(0,x)m(dx) \quad (\lambda \in (0,\lambda_{0}]). \label{QSD} 
	\end{align}
\end{Thm}

\begin{Rem} \label{rem:positivityOfDensity}
	From \cite[Corollary 5.11]{noba2023analytic}
	The density of the quasi-stationary distribution $\nu_{\lambda_{0}}$ is strictly positive; $W^{(-\lambda_{0})}(0,y) > 0 \ (y \in (0,\infty))$.
\end{Rem}

In the entrance case, the existence of the Yaglom has been shown.
Define
\begin{align}
	Z^{(q)}(x) := \lim_{y \to \infty}Z^{(q)}(x,y) = 1 + q \int_{(x,\infty)}W^{(q)}(x,u)m(du). \label{eq31}
\end{align}
Here we note that from Proposition \ref{prop:characterizationOfZ} and the entrance condition
\begin{align}
	\int_{(x,\infty)}|W^{(q)}(x,u)|m(du) < \infty \quad \text{and} \quad |Z^{(q)}(x)| \leq 1 + q\bar{W}(x,\infty)\mathrm{e}^{|q|\bar{W}(x,\infty)}. \label{barWbound}
\end{align}
Thus, the function $Z^{(q)}$ is well-defined, and it is worth emphasizing that $Z^{(q)}$ is bounded.
It is easy to see that the function $q \mapsto Z^{(q)}(x)$ is entire from Theorem \ref{thm:Zexpansion} and Proposition \ref{prop:characterizationOfZ}.

The function $Z^{(-\lambda_{0})}$ is $\lambda_{0}$-invariant function, which is a key element of the $R$-theory (see \cite{TuominenTweedie} for details).

\begin{Prop}[{\citet[Proposition 5.19]{noba2023analytic}}] \label{prop:invariantFunction}
	Assume the boundary $\infty$ is entrance.
	The following hold:
	\begin{enumerate}
		\item The function $Z^{(-\lambda_{0})}$ is $\lambda_{0}$-invariant, that is, $Z^{(-\lambda_{0})}(x) \in (0,\infty) \ (x \in I)$ and
		\begin{align}
			p_{t}Z^{(-\lambda_{0})}(x) = \mathrm{e}^{-\lambda_{0} t}Z^{(-\lambda_{0})}(x) \quad (t \geq 0, \ x \in I). \label{eq63}
		\end{align}
		\item The function $I \ni x \to Z^{(-\lambda_{0})}(x)$ is strictly increasing and $\lim_{x \to \infty}Z^{(-\lambda_{0})}(x) = 1$.
		\item The function $\bC \ni q \mapsto Z^{(q)}(0)$ has a simple root at $q = -\lambda_{0}$ and we have
		\begin{align}
			\rho := \left. \frac{d}{dq} Z^{(q)}(0) \right|_{q = -\lambda_{0}} = \int_{(0,\infty)}W^{(-\lambda_{0})}(0,u)Z^{(-\lambda_{0})}(u)m(du) \in (0,\infty). \label{}
		\end{align}
	\end{enumerate}
\end{Prop}

We have a kind of weak ergodicity.

\begin{Thm}[{\citet[Theorem 5.20]{noba2023analytic}}] \label{thm:meanYaglomLimit}
	Assume the boundary $\infty$ is entrance.
	For a measurable function $f: I \to \bR$ such that $\int_{(0,\infty)}W^{(-\lambda_{0})}(0,u) |f(u)|m(du) < \infty$, we have
	\begin{align}
			\lim_{t \to \infty} \frac{1}{t}\int_{0}^{t}\mathrm{e}^{\lambda_{0} s} p_{s}f(x)ds 
			= \frac{Z^{(-\lambda_{0})}(x)}{\rho} \int_{(0,\infty)} W^{(-\lambda_{0})}(0,u)f(u)m(du). 
		\label{meanYaglom}
	\end{align}
\end{Thm}

With a continuity condition of the transition semigroup, we can strengthen the convergence \eqref{meanYaglom}, and show the existence of the Yaglom limit.
Let $\cB_{b}(I)$ be the set of bounded Borel measurable function on $I$ and denote $||f||_{\infty} := \sup_{x \in I}|f(x)|$.
The following was obtained by applying the $R$-theory (see \cite{TuominenTweedie}).

\begin{Thm}[{\citet[Theorem 5.24]{noba2023analytic}}] \label{thm:YaglomLimit}
	Assume the boundary $\infty$ is entrance and the following continuity holds: 
	\begin{description}
		\item[(C)] For every $A \in \cB(I)$ and $x \in I$, the function $(0,\infty) \ni t \longmapsto p_{t}1_{A}(x)$ is continuous. 
	\end{description}
	Then the following holds: for every $x \in I$
	\begin{align}
		\lim_{t \to \infty} \sup_{f \in \cB_{b}(I),||f||_{\infty} \leq 1} \left|\mathrm{e}^{\lambda_{0}t}p_{t}f(x) - \frac{Z^{(-\lambda_{0})}(x)}{\lambda_{0} \rho} \nu_{\lambda_{0}}(f)\right| = 0. \label{eq40}
	\end{align}
	In particular, the quasi-stationary distribution $\nu_{\lambda_{0}}$ is the Yaglom limit under the total variation distance, that is, for every $x \in I$,
	\begin{align}
		\lim_{t \to \infty} \sup_{f \in \cB_{b}(I),||f||_{\infty} \leq 1}\left|\bE_{x}[f(X_{t}) \mid \tau_{0} > t] - \nu_{\lambda_{0}}(f)\right| = 0. \label{eq68}
	\end{align}
\end{Thm}

We will show in Theorem \ref{thm:rateOfCovergence}, under the stronger assumption (SF) than (C), the convergence \eqref{eq68} is exponentially fast.

\section{The instantaneous entrance condition and the Feller property} \label{section:entranceBoundaryAndFellerProperty}

In this section, we recall and generalize results in \cite{entranceBdry}, which showed that under the instantaneous entrance condition and a regularity on the transition semigroup, the process can be made into a Feller process by extending the state space.

Throughout this section, we suppose that $X=(\Omega, \cF, \cF_t, X_t, \theta_t, \bP_x)$ is a standard process with no negative jump.
We denote the lifetime by $\zeta$.
Since we do not need the scale function in this section, we do not assume conditions (A1)-(A3) except for Proposition \ref{prop:twoEntranceConditions}.
Since we consider the situation where $\ell_{2}$ is inaccessible, we may assume $I = [0,\infty)$ or $(0,\infty)$.

We introduce a function space suitable for our purpose:
\begin{align}
	C := \{ f: I \to \bR \mid f \text{ is bounded continuous, and when $0 \not\in I$, $\lim_{x \to 0}f(x) = 0$} \}
\end{align}
We equip this space with the supremum norm.
We say that the boundary $\infty$ is instantaneous entrance when
\begin{align}
	\lim_{b \to \infty}\lim_{x \to \infty}\bP_{x}[\tau_{b}^{-} \leq t] = 1 \quad \text{for every $t > 0$}. \label{instantaneousEntrance}
\end{align}
Note that the limits in \eqref{instantaneousEntrance} always exists by monotonicity.
Also, note that \eqref{instantaneousEntrance} implies
\begin{align}
	\lim_{b \to \infty}\lim_{x \to \infty} \bP_{x}[\tau_{b}^{-} < \zeta] = 1. \label{RareKillingAroundInfinity}
\end{align}

We note elementary consequences of \eqref{instantaneousEntrance}.

\begin{Prop} \label{prop:entranceEquivalence}
	The following hold:
	\begin{enumerate}
		\item The distributions $(\bP_{x}[\tau_{b}^{-} \in dt])_{x > 0}$ on $[0,\infty]$ weakly converges to a distribution $\mu_{b}$ on $[0,\infty]$ as $x \to \infty$, that is, for every bounded continuous function $f$ on $[0,\infty]$ we have $\lim_{x \to \infty}\bE_{x}[f(\tau_{b}^{-})] = \int_{[0,\infty]}f(t)\mu_{b}(dt)$.
		In addition, the distributions $(\mu_{b})_{b > 0}$ on $[0,\infty]$ weakly converges to the Dirac mass at $0$ as $b \to \infty$. 
		\label{prop:entranceEquivalence(i)}
		\item For every $t \in (0,\infty)$, we have 
		\[
		\lim_{b \to \infty}\lim_{x \to \infty} \bP_{x}[\tau_{b}^{-} \leq t] \in \{0,1\}.
		\]
		\label{prop:entranceEquivalence(ii)}
		\item The condition \eqref{instantaneousEntrance} is equivalent to
		\[
		\lim_{b \to \infty}\lim_{x \to \infty}\bE_{x}[\mathrm{e}^{-q\tau_{b}^{-}}] = 1 \quad \text{for every (or some) $q > 0$}.
		\]\label{prop:entranceEquivalence(iii)}
	\end{enumerate}
\end{Prop}

\begin{proof}
	Part (i) readily follows from monotonicity and the relative compactness of probability distributions on $[0,\infty]$ since \eqref{instantaneousEntrance} identifies the limit.

	Let $b < c < x$ and let $q > 0$.
	By the strong Markov property, we have $\bE_{x}[\mathrm{e}^{-q\tau_{b}^{-}}] = \bE_{x}[\mathrm{e}^{-q\tau_{c}^{-}}]\bE_{c}[\mathrm{e}^{-q\tau_{b}^{-}}]$. 
	Taking the limits as $x \to \infty$, $c \to \infty$, and $b \to \infty$ in this order, we have $a = a^{2}$ for $a := \lim_{c \to \infty}\lim_{x \to \infty}\bE_{x}[\mathrm{e}^{-q\tau_{c}^{-}}] \in [0,1]$.
	Thus, $a = 0$ or $1$.
	By the dominated converge theorem and the identity
	\[
	\bE_{x}[\mathrm{e}^{-q\tau_{b}^{-}}] = q\int_{0}^{\infty}\mathrm{e}^{-qt} \bP_{x}[\tau_{b}^{-} \leq t] dt,
	\]
	it follows
	\begin{align}
		a = q\int_{0}^{\infty}\mathrm{e}^{-qt}\lim_{b \to \infty}\lim_{x \to \infty}\bP_{x}[\tau_{b}^{-} \leq t]dt.
	\end{align}
	Parts \eqref{prop:entranceEquivalence(ii)} and \eqref{prop:entranceEquivalence(iii)} are now obvious.	
\end{proof}

In \cite{entranceBdry}, they established for conservative strong Markov processes with condition \eqref{instantaneousEntrance} and a regularity of the transition semigroup, that the state space can be extended to $[0,\infty]$, and the transition semigroup has the Feller property on it.

\begin{Thm}[{\citet[Theorem 2.2]{entranceBdry}}] \label{thm:entranceSufficient1}
	Let $I = [0,\infty)$ and $X$ is conservative.
	Suppose \eqref{instantaneousEntrance} and the following hold:
	\begin{enumerate}
		\item For every $f \in C$, $p_{t}f(x) \in C$.
		\item For every $f \in C$ and $x \in I$, $\lim_{t \to 0}p_{t}f(x) = f(x)$.
	\end{enumerate}
	Then there exist a Feller semigroup $(\bar{p}_{t})_{t \geq 0}$ on $C([0,\infty])$ such that $\bar{p}_{t}f(x) = p_{t}f(x) \ (f \in C([0,\infty]),\  x \in I)$.
	Moreover, any associated Hunt process $Y=(\Omega', \cF', \cF'_{t}, Y_{t}, \theta'_{t}, \bP'_{x})$ on $[0,\infty]$ is conservative and instantaneously entrance to $I$ from $\infty$:
	\begin{align}
		\bP'_{\infty}[Y_{t} \in I \ \text{for every $t > 0$}] = 1. \label{neverReturn2}
	\end{align}
\end{Thm}

\begin{Rem}
	Since we are assuming that the process $X$ is a Hunt process, condition (ii) is trivially satisfied by the right-continuity of paths.
	The same is true for condition (ii) in Theorem \ref{thm:entranceSufficient2} below.
	We include condition (ii) in the statement just for clarity and consistency with the previous study \cite{entranceBdry}, where the process was only assumed to be a strong Markov process with no negative jumps, making condition (ii) essential.
	However, in \cite{entranceBdry}, since the strong Markov property at time $\tau_{b}^{-}$ was used, some regularity of paths are implicitly assumed.
	In addition, defining the  ``absence of negative jumps'' as a measurable set also appears to require a certain kind of path regularity and completion of the $\sigma$-field.
\end{Rem}

In Theorem \ref{thm:entranceSufficient1}, two regularity assumptions on the transition semigroup are imposed.
We relax them with the corresponding ones on the resolvent.
In addition, we extend Theorem \ref{thm:entranceSufficient1} to non-conservative cases.
This generalization seems beneficial since it is not rare that we only know the regularity of the resolvent.
As usual, for a topological space $S$, we denote by $C_{\infty}(S)$ the space of continuous functions vanishing at infinity.
Let $I' := I \cup \{\infty\}$ be the extended interval endowed with the usual topology.
Note that $I' = (0,\infty]$ or $[0,\infty]$.
We add to $I'$ a point $\partial'$ as an isolated point. 

\begin{Thm} \label{thm:entranceSufficient2}
	Suppose \eqref{instantaneousEntrance} and the following hold:
	\begin{enumerate}
		\item For every $f \in C$ and $q > 0$, $R^{(q)}f \in C$.
		\item For every $f \in C$ and $x \in I$, $\lim_{q \to \infty} qR^{(q)}f(x) =f(x)$.
	\end{enumerate}
	Then there exist a Feller semigroup $(\bar{p}_{t})_{t \geq 0}$ on $C_{\infty}(I')$ such that $\bar{p}_{t}f(x) = \bE_{x}[f(X_{t})] \ (f \in C_{\infty}(I'),\  x \in I)$.
	Moreover, any associated Hunt process $Y=(\Omega', \cF', \cF'_{t}, Y_t, \theta'_{t}, \bP'_{x})$ on $I' \cup \{\partial'\}$ with the cemetery $\partial'$ is instantaneously entrance to $I$ from $\infty$:
	\begin{align}
		\bP'_{\infty}[Y_{t} \in I \cup \{ \partial' \} \ \text{for every $t > 0$}] = 1. \label{neverReturn}
	\end{align}
\end{Thm}

\begin{proof}
	Let $f:I \to \bR$ be bounded measurable.
	We first show $\lim_{x \to \infty}R^{(q)}f(x)$ exists for $q > 0$.
	Fix $\eps > 0$.
	From Proposition \ref{prop:entranceEquivalence} \eqref{prop:entranceEquivalence(iii)}, we can take $b > 0$ so that $\lim_{x \to \infty}\bE_{x}[\mathrm{e}^{-q \tau_{b}^{-}}] > 1 - \eps$.
	For $x \in I_{>b}$, we have from the strong Markov property at $\tau_{b}^{-}$
	\begin{align}
		R^{(q)}f(x) = \int_{0}^{\infty}\mathrm{e}^{-qt}\bE_{x}[f(X_{t}),\tau_{b}^{-} > t]dt + \bE_{x}[\mathrm{e}^{-q\tau_{b}^{-}}]R^{(q)}f(b).
	\end{align}
	It follows for $x,y \in I_{>b}$
	\begin{align}		
	&|R^{(q)}f(x) - R^{(q)}f(y)| \\
	\leq &||f||_{\infty} \left(\frac{1 - \bE_{x}[\mathrm{e}^{-q\tau_{b}^{-}}]}{q} + \frac{1 - \bE_{y}[\mathrm{e}^{-q\tau_{b}^{-}}]}{q} \right) + |\bE_{x}[\mathrm{e}^{-q\tau_{b}^{-}}] - \bE_{y}[\mathrm{e}^{-q\tau_{b}^{-}}]| \cdot |R^{(q)}f(b)| \\
	< &\frac{3||f||_{\infty}}{q}\eps.
	\end{align}
	Thus, the limit $R^{(q)}f(\infty) := \lim_{x \to \infty}R^{(q)}f(x)$ exists.
	Combining this with the definition of $C$, we see that $R^{(q)}$ maps $C_{\infty}(I')$ to itself.

	Consider the following condition:
	\begin{align}
		\lim_{q \to \infty}qR^{(q)}f(x) = f(x) \quad \text{for $x \in I'$ and $f \in C_{\infty}(I')$.} \label{eq25}
	\end{align}
	It is well known that in this situation \eqref{eq25} implies the strong continuity of the resolvent; $\sup_{x \in I'}|qR^{(q)}f(x) - f(x)| \xrightarrow{q \to \infty} 0$ (see e.g., \cite[Proof of (6.7) in Chapter III]{RogersWilliams1}).
	Then $(R^{(q)})_{q > 0}$ turns out to be a strongly continuous contraction resolvent on $C_{\infty}(I')$, and by the Hille-Yosida theorem the corresponding Feller semigroup exists.
	Thus, we show \eqref{eq25}.
	Since \eqref{eq25} is obvious for $x \in I$ from the right-continuity of paths, we only prove for $x = \infty$.
	Fix $\eps > 0$.
	Let $f \in C_{\infty}(I')$ and take $b \in I$ so that
	\begin{align}
		|f(x)-f(\infty)| < \eps \quad \text{and} \quad \bP_{x}[\zeta > \tau_{b}^{-}] > 1 - \eps \quad \text{for every $x \in I_{>b}$}.
	\end{align}
	Using the following easily checked equality
	\[
	1 = q  \int_{0}^{\infty}\mathrm{e}^{-qt} \bP_{x}[\zeta > \tau_{b}^{-} > t]dt + \bP_{x}[\zeta \leq \tau_{b}^{-}] + \bE_{x}[\mathrm{e}^{-q\tau_{b}^{-}}],
	\]
	we have from the strong Markov property at $\tau_{b}^{-}$ that for any $\delta > 0$ and $x > b+\delta$
	\begin{align}
		&|qR^{(q)}f(x) -f(x)| \\
		\begin{split}
		\leq &q  \int_{0}^{\infty}\mathrm{e}^{-qt}\bE_{x}[ |f(X_{t}) - f(x)|, \zeta > \tau_{b}^{-} > t]dt 
		+ \bE_{x}[\mathrm{e}^{-q\tau_{b}^{-}}]\left| qR^{(q)}f(b) - f(x) \right| \\
		&+ |f(x)|\bP_{x}[\zeta \leq \tau_{b}^{-}]
		\end{split} \\
		\leq &2\eps + 2||f||_{\infty}\bE_{b+\delta}[\mathrm{e}^{-q\tau_{b}^{-}}] + ||f||_{\infty}\eps.
	\end{align}
	Thus, taking limit as $x \to \infty$, we see that
	\[
	|qR^{(q)}f(\infty) -f(\infty)| \leq (2+||f||_{\infty})\eps + 2||f||_{\infty}\bE_{b+\delta}[\mathrm{e}^{-q\tau_{b}^{-}}].
	\]
	Since $\bP_{b+\delta}[\tau_{b}^{-} > 0] = 1$ by the right-continuity, it follows that $\lim_{q \to \infty}\bE_{b+\delta}[\mathrm{e}^{-q\tau_{b}^{-}}] = 0$.
	Hence, we obtain \eqref{eq25}.

	We have the Feller semigroup $(\bar{p}_{t})_{t \geq 0}$ on $C_{\infty}(I')$ associated with $(R^{(q)})_{q > 0}$.
	There exists an associated Hunt process $Y=(\Omega', \cF', \cF'_{t}, Y_{t}, \theta'_{t}, \bP'_x)$ on $I' \cup \{\partial'\}$ (see e.g., \citet[(9,4) in Chapter 1]{BluGet1968}).
	By abuse of notation, we denote the transition operator by the same symbol $(\bar{p}_{t})_{t \geq 0}$; $\bar{p}_{t}f(x) := \bE'_{x}[f(Y_{t})]$ for non-negative measurable function $f: I' \to \bR$.
	Finally, we show \eqref{neverReturn}.
	We first check 
	\begin{align}
		\bar{p}_{t}1_{\{\infty\}}(\infty) = 0 \quad \text{for every $t > 0$}. \label{eq29}
	\end{align}
	For $b > 1$, let $f_{b}:I' \to [0,1]$ be an element of $C$ such that $f_{b}(x)= 0 \ (x \in I_{<b-1})$ and $f_{b}(x)=1 \ (x \in I_{>b} \cup \{\infty\})$.
	On the one hand, we have from the dominated convergence theorem
	\begin{align}
		R^{(q)}f_{b}(x) \xrightarrow{x \to \infty} \int_{0}^{\infty}\mathrm{e}^{-qt}\bar{p}_{t}f_{b}(\infty)dt \xrightarrow{b \to \infty} \int_{0}^{\infty}\mathrm{e}^{-qt}\bar{p}_{t}1_{\{\infty\}}(\infty)dt. \label{eq26}
	\end{align}
	On the other hand, we have for $c \in I_{<x}$ that
	\begin{align}
		R^{(q)}f_{b}(x) = &\bE_{x}[\mathrm{e}^{-q\tau_{c}^{-}}]R^{(q)}f_{b}(c)
		+ \int_{0}^{\infty}\mathrm{e}^{-qt} \bE_{x}[f_{b}(X_{t}), \tau_{c}^{-} > t]dt.
	\end{align}
	The first term vanishes as $x \to \infty$ and then $b \to \infty$ since $\bE_{x}[\mathrm{e}^{-q\tau_{c}^{-}}] \leq 1$ and 
	\begin{align}
		R^{(q)}f_{b}(c) \leq \int_{0}^{\infty}\mathrm{e}^{-qt}\bP_{c}[X_{t} > b-1]dt \xrightarrow{b \to \infty} 0.
	\end{align} 
	Thus, it follows
	\begin{align}
		\limsup_{b \to \infty}\lim_{x \to \infty}R^{(q)}f_{b}(x) \leq \int_{0}^{\infty}\mathrm{e}^{-qt}\lim_{x \to \infty}\bP_{x}[\tau_{c}^{-} > t]dt.
	\end{align}
	Since $c \in I$ can be taken arbitrarily large, we have $\lim_{b \to \infty}\lim_{x \to \infty}R^{(q)}f_{b}(x) = 0$.
	Consequently, we have from \eqref{eq26} that $\bar{p}_{t}1_{\{\infty\}}(\infty) = 0$ for almost every $t > 0$.
	This is readily extended to all $t > 0$ by the semigroup property and $\bar{p}_{t}1_{\{\infty\}}(x) = 0 \ (x \in I)$.
	Thus, we have \eqref{eq29}.
	
	Set $\tau := \inf \{ t > 0 \mid Y_{t} = \infty \}$ and assume $\bP'_{\infty}[\tau < \infty] > 0$.
	From \eqref{eq29}, the state $\infty$ cannot be a holding point; $\bP'_{\infty}[\tau_{I} = 0] = 1$.
	Thus, there exist $t > 0$ and $r > 0$ such that
	\begin{align}
			\bP'_{\infty}[\tau_{r}^{-} < t, \tau \circ \theta'_{\tau_{r}^{-}} + \tau_{r}^{-} < t] > 0. \label{eq42}
	\end{align}
	By the strong Markov property at $\tau_{r}^{-}$, we have
	\[
	\bP'_{\infty}[\tau_{r}^{-} < t, \tau \circ \theta'_{\tau_{r}^{-}} + \tau_{r}^{-} < t] \leq \bP'_{\infty}\left[\bP_{Y_{\tau_{r}^{-}}}[\tau_{\infty} < \zeta], \tau_{r}^{-} < t \right].
	\]
	Since $Y_{\tau_{r}^{-}} \in I$ on $\{\tau_{r}^{-} < t\}$, we have $\bP_{Y_{\tau_{r}^{-}}}[\tau_{\infty} < \zeta] = 0$.
	This contradicts to \eqref{eq42}, and we obtain $\bP'_{\infty}[\tau < \infty] = 0$.
	The proof is complete.
\end{proof}

\begin{Rem} \label{rem:absenceOfNegativeJumpFromInfinity}
	The process $Y$ starting from $\infty$ does not have negative jumps.
	Indeed, for every $n \in \bN$ and $x < y$, we have
	\begin{align}
		&\bP'_{\infty}[1/n < \tau_{x}^{-} < \tau_{y}, Y_{1/n} \in (y,\infty)] \\
		= &\bP'_{\infty}[\bP_{Y_{1/n}}[\tau_{x}^{-} < \tau_{y}], 1/n < \tau_{x}^{-} \wedge \tau_{y}, Y_{1/n} \in (y,\infty)]\\
		= &0.
	\end{align}
	Taking limit as $n \to \infty$, we obtain $\bP'_{\infty}[\tau_{x}^{-} < \tau_{y}] = 0$, where we note that $1_{(y,\infty)}(Y_{1/n}) \to 1$ $\bP'_{\infty}$-almost surely by the right-continuity and \eqref{neverReturn}.
\end{Rem}

\begin{Rem}
	We easily check that when the process $X$ satisfies (A1)-(A3), the process $Y$ also satisfies them.
	Note that from \eqref{neverReturn}, a reference measure $m_{Y}$ of $Y$ can be given by $m_{Y}(\cdot) := m_{X}(\cdot \cap I)$, where $m_{X}$ is the reference measure of $X$. 
\end{Rem}


For later use, we present a regularity of the transition semigroup, which was essentially shown in \cite[Proof of Theorem 2.2]{entranceBdry}.

\begin{Prop} \label{prop:limitAtInfinity}
	Suppose \eqref{instantaneousEntrance} holds.
	Then for every bounded continuous function $f:I \to \bR $ and $t > 0$, there exists limit $\lim_{x \to \infty}p_{t}f(x) \in \bR$.
\end{Prop}

\begin{proof}
	Fix $t > 0$.
	Set $g(b,u) := p_{(t-u)\vee 0 }f(b)$.
	By the right-continuity and quasi-left-continuity of paths, the function $[0,\infty] \ni u \mapsto g(b,u)$ is bounded continuous.
	From the strong Markov property, we have for $x > b$
	\begin{align}
		p_{t}f(x) &= \bE_{x}[f(X_{t}),\tau_{b}^{-} > t] + \bE_{x}[p_{t-\tau_{b}^{-}}f(b),\tau_{b} \leq t] \\
		&= \bE_{x}[f(X_{t}),\tau_{b}^{-} > t] + \bE_{x}[g(b,\tau_{b}^{-})] - f(b)\bP_{x}[\tau_{b}^{-} > t].
	\end{align}
	Thus, it follows for $b < x,y$
	\begin{align}
		|p_{t}f(x)-p_{t}f(y)| \leq 2||f||_{\infty}(\bP_{x}[\tau_{b}^{-} > t] + \bP_{y}[\tau_{b}^{-} > t]) + |\bE_{x}[g(b,\tau_{b}^{-})] - \bE_{y}[g(b,\tau_{b}^{-})]|.
	\end{align}
	Each term in the RHS vanishes as $x,y \to \infty$ and then $b \to \infty$ from \eqref{instantaneousEntrance} and Proposition \ref{prop:entranceEquivalence} \eqref{prop:entranceEquivalence(i)}, respectively.
	The proof is complete.
\end{proof}

When the condition \eqref{instantaneousEntrance} fails,
the process $X$ is a Feller process on $I$ under the same regularity on the resolvent without extending the state space.

\begin{Prop}
	Suppose \eqref{instantaneousEntrance} does not hold.
	Assume (i), (ii) in Theorem \ref{thm:entranceSufficient2}.
	Then the semigroup $(p_{t})_{t \geq 0}$ is a Feller semigroup on $C_{\infty}(I)$.
\end{Prop}

\begin{proof}
	To check that $(R^{(q)})_{q > 0}$ is a strongly continuous contraction resolvent on $C_{\infty}(I)$, it suffices to show that $\lim_{x \to \infty}R^{(q)}f(x) = 0 \ (q > 0, \ f\in C_{\infty}(I))$.
	Let $0 < b < x$.
	From the strong Markov property at $\tau_{b}^{-}$ we have
	\begin{align}
		|R^{(q)}f(x)| \leq \sup_{y \in I}|f(y)|\int_{0}^{\infty}\mathrm{e}^{-qt}\bP_{x}[\tau_{b}^{-} \leq t]dt + q^{-1}\sup_{y > b}|f(y)|.
	\end{align}
	The RHS goes to $0$ as $x \to \infty$ and then $b \to \infty$ from the dominated convergence theorem and Proposition \ref{prop:entranceEquivalence} \eqref{prop:entranceEquivalence(ii)}.
\end{proof}

Finally, we present an equivalent condition to \eqref{instantaneousEntrance}.

\begin{Prop} \label{prop:twoEntranceConditions}
	Let the scale function exist, that is, (A1)-(A3) hold.
	Then \eqref{instantaneousEntrance} holds if and only if \eqref{entranceCondition} and \eqref{RareKillingAroundInfinity} hold.
\end{Prop}

\begin{Rem}
	The condition \eqref{RareKillingAroundInfinity} is apparently similar to \eqref{instantaneousEntrance}, but they are quite different.
	Roughly speaking, the condition \eqref{instantaneousEntrance} states that the process leaves the neighborhood of infinity quickly without being killed, whereas the condition \eqref{RareKillingAroundInfinity} means that killing rarely happens around the neighborhood of infinity.
	For example, any process on $(0,\infty)$ with $\bP_{x}[\tau_{0} < \infty] = 1 \ (x > 0)$ satisfies \eqref{RareKillingAroundInfinity} regardless of whether it satisfies \eqref{instantaneousEntrance}.
\end{Rem}

\begin{proof}[Proof of Proposition \ref{prop:twoEntranceConditions}]
	Let $q > 0$.
	First, note that since 
	\[
	W(b,z) \leq W^{(q)}(b,z) \leq W(b,z) \mathrm{e}^{q\bar{W}(b,z)}
	\]
	for $q \geq 0$ and $z > b$ from \eqref{eq43} and \eqref{eq01}, we see that $\bar{W}(b,\infty) < \infty$ is equivalent to $\int_{(b,\infty)}W^{(q)}(b,u)m(du) < \infty$.
	From Theorem \ref{thm:potentialDensity}, we have for $b < x < z$
	\begin{align}
		\frac{1 - \bE_{x}[\mathrm{e}^{-q(\tau_{b}^{-} \wedge \tau_{z}^{+}\wedge \zeta)}]}{q} = \int_{(b,z)}\left(\bE_{x}[\mathrm{e}^{-q\tau_{b}^{-}},\tau_{b}^{-} < \tau_{z}^{+}]W^{(q)}(b,u) - W^{(q)}(x,u)\right)m(du). \label{eq44}
	\end{align}

	Suppose \eqref{instantaneousEntrance}.
	The condition \eqref{RareKillingAroundInfinity} is obviously satisfied.
	From \eqref{eq44}, it follows
	\begin{align}
		\frac{1}{q} \geq &\int_{(b,z)}\left(\bE_{x}[\mathrm{e}^{-q\tau_{b}^{-}},\tau_{b}^{-} < \tau_{z}^{+}]W^{(q)}(b,u) - W^{(q)}(x,u)\right)m(du) \\
		\geq &\bE_{x}[\mathrm{e}^{-q\tau_{b}^{-}},\tau_{b}^{-} < \tau_{z}^{+}]\int_{(b,x)} W^{(q)}(b,u) m(du),
	\end{align}
	where we used $W^{(q)}(x,u) = 0 \ (u < x)$.
	From Proposition \ref{prop:entranceEquivalence} \eqref{prop:entranceEquivalence(iii)}, we can take a large $b_{0} > 0$ such that $\lim_{x \to \infty}\bE_{x}[\mathrm{e}^{-q\tau_{b_{0}}^{-}}] > 0$.
	Taking limit as $z \to \infty$ and then $x \to \infty$, we have from the monotone convergence theorem
	\[
	\frac{1}{q} \geq \lim_{x \to \infty}\bE_{x}[\mathrm{e}^{-q\tau_{b_{0}}^{-}}]\int_{(b_{0},\infty)} W^{(q)}(b_{0},u) m(du).
	\]
	Hence, we obtain $\int_{(b_{0},\infty)} W^{(q)}(b_{0},u) m(du) < \infty$.

	Suppose \eqref{entranceCondition} and \eqref{RareKillingAroundInfinity}.
	First, we consider the RHS of \eqref{eq44}.
	Note that $I_{<u} \ni x \mapsto W^{(q)}(x,u)$ is non-increasing from Proposition \ref{prop:exitProblem1}.
	The dominated convergence theorem gives
	\begin{align}
		&\lim_{b \to \infty}\lim_{x \to \infty}\lim_{z \to \infty}\int_{(b,z)}\left(\bE_{x}[\mathrm{e}^{-q\tau_{b}^{-}},\tau_{b}^{-} < \tau_{z}^{+}]W^{(q)}(b,u) - W^{(q)}(x,u)\right)m(du) \\
		=&\lim_{b \to \infty}\lim_{x \to \infty} \left(\bE_{x}[\mathrm{e}^{-q\tau_{b}^{-}}]\int_{(b,\infty)} W^{(q)}(b,u) m(du) - \int_{(b,\infty)}W^{(q)}(x,u)m(du)\right) \\
		=&\lim_{b \to \infty} \left( \lim_{x \to \infty} \bE_{x}[\mathrm{e}^{-q\tau_{b}^{-}}] \int_{(b,\infty)} W^{(q)}(b,u) m(du) \right) \\
		=&0.
	\end{align}
	Thus, we have from \eqref{eq44} that
	\[
	1 = \lim_{b \to \infty}\lim_{x \to \infty}\lim_{z \to \infty}\bE_{x}[\mathrm{e}^{-q(\tau_{b}^{-} \wedge \tau_{z}^{+}\wedge \zeta)}] = \lim_{b \to \infty}\lim_{x \to \infty}\bE_{x}[\mathrm{e}^{-q(\tau_{b}^{-} \wedge \zeta)}].
	\]
	It is enough to show $\lim_{b \to \infty}\lim_{x \to \infty}\bE_{x}[\mathrm{e}^{-q(\tau_{b}^{-} \wedge \zeta)}] = \lim_{b \to \infty}\lim_{x \to \infty}\bE_{x}[\mathrm{e}^{-q\tau_{b}^{-}}]$.
	We have
	\[
	\bE_{x}[\mathrm{e}^{-q(\tau_{b}^{-} \wedge \zeta)}] = \bE_{x}[\mathrm{e}^{-q\tau_{b}^{-}}] + \bE_{x}[\mathrm{e}^{-q \zeta},\zeta < \tau_{b}^{-}],
	\]
	where we note that $\tau_{b}^{-} = \infty$ on $\{ \zeta < \tau_{b}^{-} \}$.
	The second term in the RHS vanishes as $x \to \infty$ and then $b \to \infty$ from \eqref{RareKillingAroundInfinity}.
	The proof is complete.
\end{proof}

\section{Spectrum of the infinitesimal generator and compactness of the resolvent} \label{section:spectralTheory}

In the previous section, we showed that a process satisfying \eqref{instantaneousEntrance} can be extended to a Feller process under a certain regularity on the resolvent. 
In this section, we study the spectrum of the infinitesimal generator.

Throughout this section, we assume that $I = (0,\infty)$ and (A1)-(A3) hold so that the scale function exists.
In addition, we suppose that \eqref{instantaneousEntrance}, $\bP_{x}[\tau_{0} < \infty] > 0 \ (x > 0)$ and as an essential assumption we assume
\begin{align}
	\zeta = \tau_{0}. \label{killedAt0}
\end{align}
We first show that under these conditions the process always satisfies the assumptions of Theorem \ref{thm:entranceSufficient2}, and hence the process can be extended to a Feller process.
Next, we study the spectrum of the infinitesimal generator.
We show the spectrum coincides with the zeros of the function $q \mapsto Z^{(q)}(0)$ and is composed of eigenvalues.
In what follows, we use the same symbol for the complexification of the function spaces and linear operators on them that we have already introduced, as no confusion is likely to arise.


A crucial consequence of the assumption \eqref{killedAt0} is that the potential density is a meromorphic function on $\bC$.
From \eqref{eq54}, we have for $ 0 \leq y < z$
\begin{align}
	\frac{\bE_{y}[\mathrm{e}^{-q\tau_{z}^{+}}, \tau_{z}^{+} < \tau^-_{0}]}{Z^{(q)}(0,z)} = \frac{Z^{(q)}(y,z)}{Z^{(q)}(0,z)} - \frac{W^{(q)}(y,z)}{W^{(q)}(0,z)}.
\end{align}
Taking the limit as $z \to \infty$ it follows
\begin{align}
	\frac{Z^{(q)}(y)}{Z^{(q)}(0)} = \lim_{z \to \infty}\frac{W^{(q)}(y,z)}{W^{(q)}(0,z)} = \bE_{y}[\mathrm{e}^{-q\tau_{0}},\tau_{0} < \infty],
\end{align}
where $Z^{(q)}(\cdot)$ was defined in \eqref{eq31}.
Plugging this into \eqref{potentialDensity}, we have for $x \in (0,\infty)$ and bounded measurable function $f:(0,\infty) \to \bC$
\begin{align}
	\int_{0}^{\infty}\mathrm{e}^{-qt}p_{t}f(x)dt  = \int_{(0,\infty)} f(u)r^{(q)}(x,u)m(du) \quad (q > 0) \label{potentialDensityOneside}
\end{align}
for
\begin{align}
	r^{(q)}(x,u) := \frac{Z^{(q)}(x)}{Z^{(q)}(0)}W^{(q)}(0,u) - W^{(q)}(x,u). \label{ResolventDensityKilledAt0}
\end{align}
That is, the function $r^{(q)}$ is a potential density of $(p_{t})_{t \geq 0}$.
Denote by $\cN$ the zeros of the function $\bC \ni q \mapsto Z^{(q)}(0)$:
\begin{align}
	\cN := \{ q \in \bC \mid Z^{(q)}(0) = 0 \}. 
\end{align}
The function $q \mapsto r^{(q)}(x,u)$ is a meromorphic function on $\bC$ for each fixed $x,u \in (0,\infty)$ with the poles in $\cN$.
For $q \in \bC \setminus \cN$, we define the integral operator $R^{(q)}f(x) := \int_{(0,\infty)}r^{(q)}(x,y)f(y)m(dy) \ (x \in (0,\infty))$ for any measurable function $f$ with
\begin{align}
	\int_{(0,\infty)}|r^{(q)}(x,u)f(u)|m(du) < \infty. \label{}
\end{align}
Applying the identities \eqref{resolventIdentityW} and \eqref{resolventIdentityZ}, we can easily check that the density $r^{(q)}$ satisfies the following identity:
\begin{align}
	r^{(q)}(x,y) - r^{(q')}(x,y) = (q'-q)\int_{(0,\infty)}r^{(q)}(x,u)r^{(q')}(u,y)m(du) \quad (q,q' \in \bC \setminus \cN). \label{resolventIdentityDensity}
\end{align}
This obviously implies the following resolvent identity:
\begin{align}
	R^{(q)} - R^{(q')} = (q'-q)R^{(q)}R^{(q')} \ (q,q' \in \bC \setminus \cN) \label{resolventIdentityOp}
\end{align}
whenever the both sides are defined on a common function space.
We first show that $(R^{(q)})_{q \in \bC \setminus \cN}$ are bounded operators  on $C_{\infty}(0,\infty]$.

\begin{Prop} \label{prop:boundednessOfResolvent2}
	For $q \in \bC \setminus \cN$ and a measurable function $f:(0,\infty) \to \bC$ such that $\int_{(0,\infty)}W(0,u)|f(u)|m(du) < \infty$, we have
	\begin{align}
		\int_{(0,\infty)}|r^{(q)}(x,u)f(u)|m(du) \leq  \left(1 + \left| \frac{Z^{(q)}(x)}{Z^{(q)}(0)}\right| \right)\mathrm{e}^{|q|\bar{W}(0,\infty)}\int_{(0,\infty)}W(0,u)|f(u)|m(du), \label{Eq06-2}
	\end{align}
	and thus the function $R^{(q)}f:(0,\infty) \to \bC$ is well-defined.
	We also have $R^{(q)}f \in C_{\infty}(0,\infty]$ and
	\begin{align}
		\lim_{x \to \infty}R^{(q)}f(x) = \frac{1}{Z^{(q)}(0)}\int_{(0,\infty)}W^{(q)}(0,u)f(u)m(du). \label{Eq01-2}
	\end{align}
	Moreover, when $f \in C_{\infty}(0,\infty]$, the following holds:
	\begin{align}
		||R^{(q)}f||_{\infty} \leq \left(1 +  \frac{||Z^{(q)}||_{\infty}}{Z^{(q)}(0)} \right)\mathrm{e}^{|q|\bar{W}(0,\infty)}\bar{W}(0,\infty)||f||_{\infty}. \label{Eq06-3}
	\end{align}
	Consequently, $R^{(q)}$ is a bounded operator on $C_{\infty}(0,\infty]$.
\end{Prop}

\begin{proof}
	Recall that from \eqref{entranceCondition}, \eqref{barWbound}, and $\bP_{x}[\tau_{0} < \infty] > 0$, we have $\bar{W}(0,\infty) < \infty$ and $||Z^{(q)}||_{\infty} < \infty$.
	It follows
	\begin{align}
		\int_{(0,\infty)}|r^{(q)}(x,u)f(u)|m(du) 
		\leq \left(1 + \left| \frac{Z^{(q)}(x)}{Z^{(q)}(0)}\right| \right)\mathrm{e}^{|q|\bar{W}(0,\infty)}\int_{(0,\infty)}W(0,u)|f(u)|m(du) < \infty. \label{eq32}
	\end{align}
	Thus, the function $R^{(q)}f$ is well-defined and bounded.
	Applying \eqref{barWbound} and the dominated convergence theorem to 
	\begin{align}
		R^{(q)}f(x) = \frac{Z^{(q)}(x)}{Z^{(q)}(0)}\int_{(0,\infty)}W^{(q)}(0,u)f(u)m(du) - \int_{(0,\infty)}W^{(q)}(x,u)f(u)m(du), \label{}
	\end{align}
	we have the desired continuity of $R^{(q)}f$ and \eqref{Eq01-2}.
\end{proof}

From Proposition \ref{prop:boundednessOfResolvent}, we see that $(R^{(q)})_{q > 0}$ satisfies the assumption of Theorem \ref{thm:entranceSufficient2}.
Therefore, without loss of generality, we may assume that the process $X$ is given as a Feller process on $(0,\infty]$ from the beginning. Henceforth, $X$ will be regarded as such a process, and we denote the transition semigroup on $(0,\infty]$ by $(p_{t})_{t \geq 0}$.
By the Hille-Yosida theorem, there exists a infinitesimal generator $A$ of $(p_{t})_{t \geq 0}$ on $C_{\infty}(0,\infty]$.
We denote the spectrum of $A$ by $\sigma(A)$ and the resolvent set by $\rho(A)$.

We show that $\sigma(A) = \cN$ and $(R^{(q)})_{q \in \bC \setminus \cN}$ is the resolvent of $A$.
For linear spaces $\cX$ and $\cY$, we denote the kernel and the range of a linear operator $B:\cX \to \cY$ by $\cN(B)$ and $\cR(B)$, respectively, and denote the identity map on $\cX$ by $\mathrm{Id}_{\cX}$.

\begin{Thm} \label{thm:SpectrumIsZeros2}
	We have $\sigma(A) = \cN$ and $(q - A)^{-1} = R^{(q)} \ (q \in \rho(A))$.
	Moreover, for $q \in \sigma(A)$, we have $Z^{(q)} \in \cD(A)$ and $AZ^{(q)} = q Z^{(q)}$, that is, $\sigma(A)$ comprises eigenvalues.
\end{Thm}

\begin{proof}
	Since $(0,\infty) \subset \rho(A) \cap (\bC \setminus \cN)$, we see for $q > 0$ that $\cR(R^{(q)})$ is dense in $C_{\infty}(0,\infty]$, and $\cN(R^{(q)}) = \{0\}$.
	From Proposition \ref{prop:boundednessOfResolvent2} and \eqref{resolventIdentityOp}, the operators $(R^{(q)})_{q \in \bC \setminus \cN}$ are bounded and satisfy the resolvent identity.
	Note that such operators are called pseudo-resolvent (see e.g., \cite[Section 1.9]{MR710486}).
	It is easily observed from \eqref{resolventIdentityOp} that $\cR(R^{(q)})$ and $\cN(R^{(q)})$ do not depend on $q \in \bC \setminus \cN$.
	Thus, when $q \not\in \cN$, the inverse operator $(R^{(q)})^{-1}$ of $R^{(q)}$ exists.
	Again from \eqref{resolventIdentityOp}, it follows that the operator $q - (R^{(q)})^{-1}$ do not depend on $q \in \bC \setminus \cN$.
	Considering the case of $q > 0$, we see that $A = q - (R^{(q)})^{-1} \ (q \not\in \cN)$.
	Thus, for $q \in \bC \setminus \cN$ the operator $q-A$ has a bounded inverse, which shows $\sigma(A) \subset \cN$.
	We show that every $q \in \cN$ is an eigenvalue.
	Let $q \in \cN$ and let $r > 0$.
	From the identity \eqref{resolventIdentityZ}, it follows
	\begin{align}
		R^{(r)}Z^{(q)}(x) &= \frac{Z^{(r)}(x)}{Z^{(r)}(0)}\int_{(0,\infty)}W^{(r)}(0,u)Z^{(q)}(u)m(du) - \int_{(0,\infty)}W^{(r)}(x,u)Z^{(q)}(u))m(du) \label{} \\
		&= \frac{Z^{(r)}(x)}{Z^{(r)}(0)}\frac{Z^{(r)}(0)-Z^{(q)}(0)}{r-q} - \frac{Z^{(r)}(x) - Z^{(q)}(x)}{r-q} \label{} \\
		&= \frac{Z^{(q)}(x)}{r-q}. \label{} 
	\end{align}
	Thus, $Z^{(q)} \in \cR(R^{(r)}) = \cD(A)$.
	Since $AR^{(r)} = rR^{(r)} - \mathrm{Id}_{C_{\infty}(0,\infty]}$, it follows $A Z^{(q)}(x) = qZ^{(q)}(x)$.
	Hence, $q \in \sigma(A)$, and the proof is complete.
\end{proof}

We show the compactness of the resolvent.

\begin{Prop} \label{prop:compactness}
	For $q \in \rho(A)$, the operator $R^{(q)}$ is compact on $C_{\infty}(0,\infty]$.
\end{Prop}

\begin{proof}
	Let $(f_{n})_{n \geq 1}$ be a bounded sequence on $C_{\infty}(0,\infty]$.
	We apply the Ascoli-Arzel\`a theorem.
	The uniform boundedness readily follows from \eqref{eq32}.
	For $x,x' \geq 0$, we have from Proposition \ref{prop:defOfM} 
	\begin{align}
		|R^{(q)}f_{n}(x)-R^{(q)}f_{n}(x')| \leq & ||f_{n}||_{\infty}\bar{W}(0,\infty) \mathrm{e}^{|q|\bar{W}(0,\infty)} \frac{|Z^{(q)}(x)-Z^{(q)}(x')|}{|Z^{(q)}(0)|} \\
		&+ ||f_{n}||_{\infty} \int_{(0,\infty)}|W^{(q)}(x,u) - W^{(q)}(x',u)|m(du),
	\end{align}
	and the equicontinuity holds from the dominated convergence theorem.
\end{proof}

In the rest of this section, we study the infinitesimal generator of $(p_{t})_{t \geq 0}$ on an $L^{p}$-space ($p \in [1,\infty)$) when the underlying measure is a quasi-stationary distribution, and show that the spectra are still given by $\cN$.
For $p \in [1,\infty]$, we denote the $L^{p}$-space with respect to a Radon measure $\mu$ on $I$ by $L^{p}(\mu)$ and its norm by $||\cdot||_{L^{p}(\mu)}$.
For a given Radon measure $\mu$, it is generally difficult to determine whether the transition semigroup has an infinitesimal generator.
When a quasi-stationary distribution $\nu$ exists, however, the strong continuity on $L^{p}(\nu) \ (p \in [1,\infty))$ can be shown in a usual way, as follows.
Let us suppose a quasi-stationary $\nu$ on $I$ exists, and it satisfies $\bP_{\nu}[\tau_{0} > t] = \mathrm{e}^{-\alpha t} \ (t > 0)$ for some $\alpha > 0$.
We check that the semigroup $(p_{t})_{t \geq 0}$ is strongly continuous and contractive on $L^{p}(\nu) \ (p \in [1,\infty))$.
Since the distribution $\nu$ satisfies $\nu(p_{t}f) = \mathrm{e}^{-\alpha t}\nu(f) \ (f \in \cB_{b}(I))$ by definition, we have from Jensen's inequality for $f \in L^{p}(\nu)$ that
\begin{align}
	||p_{t}f(x)||_{L^{p}(\nu)}^{p} \leq \nu(p_{t}|f|^{p}) = \mathrm{e}^{-\alpha t}||f||_{L^{p}(\nu)}^{p}, \label{}
\end{align}
which shows $(p_{t})_{t \geq 0}$ is contractive on $L^{p}(\nu)$.
The strong continuity easily follows from the denseness of continuous functions in $L^{p}(\nu)$.
Hence, we can consider the infinitesimal generator $A_{p}$ on $L^{p}(\nu)$.
Note that since $\nu$ is a finite measure, we have $L^{p}(\nu) \subset L^{q}(\nu) \ (1 \leq q < p \leq \infty)$ and we can regard the operators $A, A_{p} \ (p \in (1,\infty))$ as restrictions of $A_{1}$.

To investigate detailed properties of $A_{p}$, we require the explicit form of the quasi-stationary distribution $\nu$.
In what follows, we assume condition (Q) in Section \ref{section:QSDprevious}.
Under this condition, as seen in Theorem \ref{thm:existenceOfQSD}, a unique quasi-stationary distribution $\nu_{\lambda_{0}}$ exists.
Henceforth, we denote by $A_{p} \ (p \in [1,\infty))$ the infinitesimal generator defined on $L^{p}(\nu_{\lambda_{0}})$.
We show that the resolvent of $A_{p}$ is still given by $R^{(q)}$.
We first check $R^{(q)} \ (q \not\in \cN)$ is a bounded operator on $L^{p}(\nu_{\lambda_{0}})$.

\begin{Prop} \label{prop:boundednessOfResolvent}
	For $q \in \bC \setminus \cN$ and $f \in L^{1}(\nu_{\lambda_{0}})$, we have
	\begin{align}
		\int_{(0,\infty)}|r^{(q)}(x,y)f(y)|m(dy) \leq \left(1 + \left| \frac{Z^{(q)}(x)}{Z^{(q)}(0)}\right| \right)\frac{\mathrm{e}^{|q+\lambda_{0}|/\lambda_{0}}}{\lambda_{0}}||f||_{L^{1}(\nu_{\lambda_{0}})} < \infty, \label{Eq06}
	\end{align}
	and the function $R^{(q)}f$ is well-defined.
	In particular, the operator $R^{(q)}$ can be defined as a bounded operator on $L^{p}(\nu_{\lambda_{0}}) \ (p \in [1,\infty])$.
	Moreover, $R^{(q)}f \in C_{\infty}(0,\infty]$ and satisfies \eqref{Eq01-2}.
\end{Prop}

\begin{proof}
	It is easy to see that if we have, for $q \in \bC$ and $x,y \in [0,\infty]$,
	\begin{align}
		W^{(-\lambda_{0})}(x,y) \leq W(x,y)
		\quad \text{and} \quad 
		|W^{(q)}(x,y)| \leq W^{(-\lambda_{0})}(x,y) \mathrm{e}^{|q+\lambda_{0}|\int_{(x,y)}W^{(-\lambda_{0})}(x,u)m(du)},
	\end{align}
	the desired result follows from almost the same argument as Proposition \ref{prop:boundednessOfResolvent2}.
	The former readily follows from Proposition \ref{prop:resolventIdentityW} and Remark \ref{rem:positivityOfDensity}.
	The latter can also be checked from Proposition \ref{prop:resolventIdentityW}.
	Indeed, it gives the expansion
	\begin{align}
		W^{(q)}(x,y) = \sum_{n \geq 0}(q+\lambda_{0})^{n}\left(W^{(-\lambda_{0})}\right)^{\otimes (n+1)}(x,y),
	\end{align}
	whose convergence is ensured by the following elementary estimate: for $n \geq 1$
	\begin{align}
		\left(W^{(-\lambda_{0})}\right)^{\otimes (n)}(x,y) \leq \frac{W^{(-\lambda_{0})}(x,y)}{(n-1)!} \left(\int_{(x,y)}W^{(-\lambda_{0})}(x,u)m(du)\right)^{n-1}. \label{Eq34}
	\end{align}
	Since \eqref{Eq34} can be routinely confirmed, we omit the detail.
\end{proof}

Once we know the operator $R^{(q)} \ (q \in \bC \setminus \cN)$ is bounded on $L^{p}(\nu_{\lambda_{0}})$ and maps $L^{p}(\nu_{\lambda_{0}})$ to $C_{\infty}(0,\infty]$, we readily obtain the following by exactly the same argument with Theorem \ref{thm:SpectrumIsZeros2}, and we omit the proof.

\begin{Prop} \label{prop:SpectrumIsZeros}
	For $p \in [1,\infty)$, we have $\sigma(A_{p}) = \cN$ and $(q - A)^{-1} = R^{(q)} \ (q \in \rho(A_{p}))$.
	Moreover, for $q \in \sigma(A_{p})$, we have $Z^{(q)} \in \cD(A_{p})$ and $A_{p}Z^{(q)} = q Z^{(q)}$, that is, $\sigma(A_{p})$ comprises eigenvalues.
\end{Prop}

The following is compactness of the resolvent on $L^{p}(\nu_{\lambda_{0}})$.
We are not certain whether it is also valid for $p=1$.

\begin{Prop} \label{prop:compactness2}
	For $q \in \bC \setminus \cN$ and $p \in (1,\infty]$, the operator $R^{(q)}$ is compact as an operator from $L^{p}$ to $C_{\infty}(0,\infty]$.
\end{Prop}

\begin{proof}
	The case $p = \infty$ follows from the same argument as Proposition \ref{prop:compactness}.
	Thus, we show the case $p \in (1,\infty)$.
	Let $(f_{n})_{n \geq 1}$ be an $L^{p}(\nu_{\lambda_{0}})$-bounded sequence.
	Since the uniform boundedness follows from \eqref{Eq06},
	it suffices to check the equicontinuity.
	From H\"older's inequality and a similar computation in Proposition \ref{prop:boundednessOfResolvent}, we have
	\begin{align}
		&|R^{(q)}f_{n}(x) - R^{(q)}f_{n}(x')| \label{} \\
		\leq &\frac{|Z^{(q)}(x) - Z^{(q)}(x')|}{|Z^{(q)}(0)|}\frac{\mathrm{e}^{|q+\lambda_{0}|/\lambda_{0}}}{\lambda_{0}}||f_{n}||_{L^{1}(\nu_{\lambda_{0}})} \label{} \\
		&+ \left(\int_{(0,\infty)}\frac{|W^{(q)}(x,u) - W^{(q)}(x',u)|^{p/(p-1)}}{W^{(-\lambda_{0})}(0,u)^{1/(p-1)}}m(du)\right)^{1-1/p}\frac{||f_{n}||_{L^{p}(\nu_{\lambda_{0}})}}{\lambda_{0}}. \label{Eq10}
	\end{align}
	We similarly see that $|W^{(q)}(x,y) - W^{(q)}(x',y)| \leq 2W^{(-\lambda_{0})}(0,y)\mathrm{e}^{|q+\lambda_{0}|/\lambda_{0}}$, which implies the integral in \eqref{Eq10} is finite.
	By the dominated convergence theorem, we have
	\begin{align}
		\lim_{x \to x'}\int_{(0,\infty)}\frac{|W^{(q)}(x,u) - W^{(q)}(x',u)|^{p/(p-1)}}{W^{(-\lambda_{0})}(0,u)^{1/(p-1)}}m(du) = 0. \label{}
	\end{align}
	Hence, the equicontinuity follows, and the proof is complete.
\end{proof}

\section{Exponential convergence to the Yaglom limit} \label{section:exponentialConvergence}

We maintain the assumptions in the previous section,
that is, we assume that (A1)-(A3), \eqref{instantaneousEntrance}, $\bP_{x}[\tau_{0} < \infty] > 0 \ (x > 0)$, and $\zeta = \tau_{0}$ hold and $X$ is a Feller process on $I = (0,\infty]$.
In this section, we study the convergence rate of the Yaglom limit in \eqref{eq68}.
The case in which $\ell_{2}$ is accessible is addressed in Appendix \ref{appendix:accessible}. 

In the rest of this section, we assume the following condition:
\begin{description}
	\item[(SF)] For every $f \in \cB_{b}(I)$ and $t > 0$, the function $(0,\infty) \ni x \mapsto p_{t}f(x)$ is bounded and continuous.
\end{description}

Note that in this condition, the continuity of $p_{t}f(x)$ at $x = \infty$ is not required, while it automatically follows. 

\begin{Prop} \label{prop:strongerRegularity}
	For every $f \in \cB_{b}(I)$ and $t > 0$, we have $p_{t}f \in C_{\infty}(0,\infty]$. 
\end{Prop}

\begin{proof}
	Fix $t > 0$.
	We first show $\lim_{x \to 0}p_{t}f(x) = 0$.
	It suffices to show that $\lim_{x \to 0}p_{t}1(x) = 0$.
	The function $x \mapsto p_{t}1(x)$ is non-decreasing by the absence of negative jumps.
	For $q > 0$ and $s > 0$, the dominated convergence theorem gives
	\begin{align}
		\int_{0}^{\infty}\mathrm{e}^{-qs}\lim_{x \to 0}p_{s}1(x)ds = \int_{(0,\infty)}\lim_{x \to 0}\left(\frac{Z^{(q)}(x)}{Z^{(q)}(0)}W^{(q)}(0,u) - W^{(q)}(x,u)\right) m(du) = 0,
	\end{align}
	and we have $\lim_{x \to 0}p_{s}1(x) = 0$ for almost every $s \in (0,\infty)$, and this is extended to every $s \in (0,\infty)$ by monotonicity.

	Set $g(x) := p_{t/2}f(x) \ (x \in (0,\infty))$.
	From the above argument, (SF), and Proposition \ref{prop:limitAtInfinity}, we see that $g \in C_{\infty}(0,\infty]$.
	From the Feller property, it follows $p_{t}f = p_{t/2}g \in C_{\infty}(0,\infty]$.
\end{proof}

To ensure the existence of a quasi-stationary distribution, we use condition (Q).
Note that the latter condition in (Q) is automatically satisfied since it follows from $\zeta = \tau_{0}$, \eqref{potentialDensityOneside} and Proposition \ref{prop:propertiesOfScaleFunc} (iii) that
\begin{align}
	\bE_{x}[\tau_{0}] = \int_{(0,\infty)}W(0,u)m(du) - \int_{(0,\infty)}W(x,u)m(du) < \infty.
\end{align}
Thus, we assume the former condition $\bP_{x}[\tau_{y} < \infty] \ (x,y \in (0,\infty))$ of (Q) in the rest of this section.
Under the condition, from Theorem \ref{thm:existenceOfQSD} the decay parameter $\lambda_{0}$ in \eqref{decayParameter} is strictly positive, and 
there exists a unique quasi-stationary distribution $\nu_{\lambda_{0}}$ in \eqref{QSD}.
Also, note that $-\lambda_{0} \in \cN$ from Proposition \ref{prop:invariantFunction} (iii).

The following is the main result of this section, which shows that the convergence to the Yaglom limit \eqref{eq68} is exponentially fast.
Moreover, the convergence is uniform with respect to a class of initial distributions (see Remark \ref{rem:uniformConvergence}).

\begin{Thm} \label{thm:rateOfCovergence}
	Let $\lambda_{1} := -\sup\{ \Re q \mid q \in \cN \setminus \{-\lambda_{0} \} \}$.
	Then the spectral gap exists: $\lambda_{1} > \lambda_{0}$.
	For any $\delta \in (0,\lambda_{1}-\lambda_{0})$, there exists $M_{\delta} \in (0,\infty)$ and we have
	\begin{align}
		\sup_{f \in \cB_{b}(I), ||f||_{\infty} \leq 1}|| \mathrm{e}^{\lambda_{0}t} p_{t}f - \pi_{\lambda_{0}}f ||_{\infty} \leq M_{\delta}\mathrm{e}^{-\delta t} \quad (t > 0), \label{Eq11}
	\end{align}
	where $\pi_{\lambda_{0}}$ is the projection defined by
	\begin{align}
		\pi_{\lambda_{0}}f(x) := \frac{1}{\rho\lambda_{0}}Z^{(-\lambda_{0})}(x)\int_{I}f(y)\nu_{\lambda_{0}}(dy). \label{}
	\end{align}
	This implies that the convergence to the Yaglom limit is exponentially fast in the total variation distance: for every distribution $\mu$ on $(0,\infty]$
	\begin{align}
		\sup_{f \in \cB_{b}(I), ||f||_{\infty} \leq 1}|\bE_{\mu}[f(X_{t}) \mid \tau_{0} > t] - \nu_{\lambda_{0}}(f)| \leq
		\frac{2M_{\delta}\mathrm{e}^{-\delta t}}{\mu(Z^{(-\lambda_{0})})} \quad (t > 0). \label{Eq12} 
	\end{align}
\end{Thm}

Under (SF), the proof is given by simply combining a few elements from general theories.
The following lemma is essential, whose proof is just a slight modification of a result in \cite{SchillingWang}.

\begin{Lem} \label{lem:compactnessOfSemigroup}
	For every $t > 0$, the operator $p_{t}: C_{\infty}(0,\infty] \to C_{\infty}(0,\infty]$ is compact. 
\end{Lem}

\begin{proof}
	In this proof, we consider the restriction of the reference measure $m$ to the Borel sets $\cB(I)$ and, for the sake of brevity, denote it by the same symbol $m$.

	Let $N \in \cB(I)$ with $m(N) = 0$.
	By the definition of the reference measure, it obviously follows $R^{(1)}1_{N}(x) = 0$ for every $x \in (0,\infty]$.
	For each $x \in (0,\infty]$, this implies $p_{t}1_{N}(x) = 0$ for almost every $t > 0$.
	Since the semigroup $(p_{t})_{t \geq 0}$ is strongly continuous on $C_{\infty}(0,\infty]$, we see that $p_{t}1_{N}(x) = 0$ for every $t > 0$ and $x \in (0,\infty]$ by Proposition \ref{prop:strongerRegularity} and (SF).
	Thus, for every $t > 0$ the subprobability $p_{t}(x,du)$ is absolutely continuous with respect to $m$ on $(0,\infty]$.
	We denote the Radon-Nikodym derivative by $\rho_{t}(x,u)$: $p_{t}(x,du) = \rho_{t}(x,u)m(du) \ (x \in (0,\infty],\ u \in (0,\infty])$.
	For $f \in L^{\infty}(m)$, define $P_{t}f(x) := \int_{(0,\infty]}\rho_{t}(x,u)f(u)m(du) \ (x \in (0,\infty])$.
	Set $S_{f} := \{x \in (0,\infty] \mid f(x) \leq ||f||_{L^{\infty}(m)} \}$.
	Since $S_{f}$ is measurable, we have $\tilde{f} := f1_{S_{f}} \in \cB_{b}(I)$.
	Then it clearly follows $P_{t}f(x) = p_{t}\tilde{f}(x)$, and $P_{t}f \in C_{\infty}(0,\infty]$ from Proposition \ref{prop:strongerRegularity}.

	We show the compactness of the operator $P_{t}:L^{\infty}(m) \to C_{\infty}(0,\infty]$.
	This obviously implies the desired result since $C_{\infty}(0,\infty] \subset L^{\infty}(m)$ and $P_{t}f = p_{t}f$ for $f \in C_{\infty}(0,\infty]$, where we note that $m$ has a full support on $(0,\infty)$ from Proposition \ref{Prop101}.
	By the Banach-Alaoglu theorem, the unit ball $U := \{ f \in L^{\infty}(m) \mid ||f||_{L^{\infty}(m)} \leq 1 \}$ of $L^{\infty}(m)$ is compact in $\sigma(L^{\infty}(m),L^{1}(m))$-topology.
	Thus, for every sequence $(f_{n})_{n \geq 0}$ in $U$, there exist a subsequence $(f_{n_{k}})_{k \geq 0}$ and $f \in U$ such that $\lim_{k \to \infty}\int_{(0,\infty]}f_{n_{k}}(u)g(u)m(du) = \int_{(0,\infty]}f(u)g(u)m(du)$ for every $g \in L^{1}(m)$.
	In particular, we have for every $x \in (0,\infty]$
	\begin{align}
		P_{t}f_{n_{k}}(x) = \int_{(0,\infty)}\rho_{t}(x,u)f_{n_{k}}(u)m(du) \xrightarrow[k \to \infty]{} \int_{(0,\infty)}\rho_{t}(x,u)f(u)m(du) = P_{t}f(x). \label{} 
	\end{align}
	Since $\lim_{x \to 0}P_{t}f_{n_{k}}(x) = \lim_{x \to 0}P_{t}f(x) = 0$, we see that the continuous functions $(P_{t}f_{n_{k}})_{k\geq 0}$ on $[0,\infty]$ converge pointwise to $P_{t}f$. 
	Fix $N > 0$ and let $k,l \geq N$.
	We have
	\begin{align}
		|P_{t}f_{n_{k}}(x) - P_{t}f_{n_{l}}(x)| \leq p_{t/2}|P_{t/2}(f_{n_{k}}-f_{n_{l}})|(x) \leq p_{t/2}h_{N}(x) \quad (x \in [0,\infty]), \label{Eq13}
	\end{align}
	where $h_{N}(x) := \sup_{i,j \geq N}|P_{t/2}(f_{n_{i}}-f_{n_{j}})(x)| \ (x \in [0,\infty])$.
	The function $h_{N}$ decreases to zero pointwise as $N \to \infty$, and the same is true for $p_{t/2}h_{N}$ by the dominated convergence theorem.
	Since the function $p_{t/2}h_{N}$ is continuous on $[0,\infty]$, Dini's theorem shows $p_{t/2}h_{N}$ converges to zero uniformly.
	Hence, from \eqref{Eq13} we have the uniform convergence of $(P_{t}f_{n_{k}})_{k \geq 0}$ as $k \to \infty$, and the proof is complete.
\end{proof}

Once we have the compactness of the transition semigroup, 
we can readily prove Theorem \ref{thm:rateOfCovergence} by applying general results on the semigroup theory.
The main reference is \cite{MR839450} (see also \cite[Appendix]{KolbSavov}).

\begin{proof}[Proof of Theorem \ref{thm:rateOfCovergence}]
	Since the residue of the resolvent density $r^{(q)}(x,y)$ at $q = -\lambda_{0}$ is
	\begin{align}
		\lim_{q \to -\lambda_{0}}(q+\lambda_{0})r^{(q)}(x,y) = \frac{Z^{(-\lambda_{0})}(x)W^{(-\lambda_{0})}(0,y)}{\rho}, \label{}
	\end{align}
	we can easily see by the dominated convergence theorem that the residue at $q = -\lambda_{0}$ of the resolvents $(R^{(q)})_{q \in \bC \setminus \cN}$ as operators on $C_{\infty}(0,\infty]$ is the projection $\pi_{\lambda_{0}}$, that is, $(q + \lambda_{0})R^{(q)}$ converges to $\pi_{\lambda_{0}}$ as $q \to -\lambda_{0}$ in operator norm.
	Since the operator $p_{t} \ (t > 0)$ is compact on $C_{\infty}(0,\infty]$ from Lemma \ref{lem:compactnessOfSemigroup}, we see from \cite[Corollary 2.2 in Chapter B-IV]{MR839450} and its proof that
	$\lambda_{1} > \lambda_{0}$ and for every $\delta \in (0,\lambda_{1}-\lambda_{0})$ there exists a constant $\tilde{M}_{\delta} > 0$ such that
	\begin{align}
		\sup_{g \in C_{\infty}(0,\infty], ||g||_{\infty} \leq 1}|| \mathrm{e}^{\lambda_{0}t} p_{t}g - \pi_{\lambda_{0}}g ||_{\infty} \leq \tilde{M}_{\delta}\mathrm{e}^{-\delta t} \quad (t \geq 0). \label{} 
	\end{align}
	Take $\eps > 0$.
	From the strong Feller property, we see for $t > \eps$ that
	\begin{align}
		\sup_{f \in \cB_{b}(I), ||f||_{\infty} \leq 1}|| \mathrm{e}^{\lambda_{0}t} p_{t}f - \pi_{\lambda_{0}}f ||_{\infty} 
		&\leq \mathrm{e}^{\lambda_{0}\eps}\sup_{g \in C_{\infty}(0,\infty], ||g||_{\infty} \leq 1}|| \mathrm{e}^{\lambda_{0}(t-\eps)} p_{t-\eps}g - \pi_{\lambda_{0}}g ||_{\infty} \label{} \\
		&\leq \tilde{M}_{\delta}\mathrm{e}^{(\lambda_{0}+\delta)\eps}\mathrm{e}^{-\delta t}, \label{}
	\end{align}
	where we used the fact that $\pi_{\lambda_{0}}(p_{\eps}f) = \mathrm{e}^{-\lambda_{0}\eps}\pi_{\lambda_{0}}f$ since $\nu_{\lambda_{0}}$ is a quasi-stationary distribution.
	Thus, we obtain \eqref{Eq11} for some $M_{\delta} > 0$.
	We show \eqref{Eq12}.
	Take a distribution $\mu$ on $(0,\infty]$ and take $f \in \cB_{b}(I)$ with $||f||_{\infty} \leq 1$.
	Noting that $\pi_{\lambda_{0}}(f) = \nu_{\lambda_{0}}(f)\pi_{\lambda_{0}}(1)$, we have
	\begin{align}
		&|\bE_{\mu}[f(X_{t}) \mid \tau_{0} > t] - \nu_{\lambda_{0}}(f)| \label{} \\
		=&\left|\frac{\mu(p_{t}f)}{\mu(p_{t}1)} - \nu_{\lambda_{0}}(f)\right| \label{} \\
		=& \frac{1}{\mu(\mathrm{e}^{\lambda_{0}t} p_{t}1)}\left|\mu(\mathrm{e}^{\lambda_{0}t} p_{t}f - \pi_{\lambda_{0}}f) - \nu_{\lambda_{0}}(f)\mu(\mathrm{e}^{\lambda_{0}t}p_{t}1 - \pi_{\lambda_{0}}(1))  \right| \label{} \\
		\leq& \frac{2M_{\delta}\mathrm{e}^{-\delta t}}{\mu(\mathrm{e}^{\lambda_{0}t} p_{t}1)} \label{} \\
		\leq& \frac{2M_{\delta}\mathrm{e}^{-\delta t}}{\mu(Z^{(-\lambda_{0})})},
	\end{align}
	where in the last inequality we used the inequality $\mathrm{e}^{\lambda_{0}t}p_{t}1(x) \geq \mathrm{e}^{\lambda_{0}t}p_{t}Z^{(-\lambda_{0})}(x) = Z^{(-\lambda_{0})}(x) > 0$ following from Proposition \ref{prop:invariantFunction}.
\end{proof}

\begin{Rem} \label{rem:uniformConvergence}
	In the RHS of \eqref{Eq12}, only $\mu(Z^{(-\lambda_{0})})$ depends on the initial distribution $\mu$.
	Thus, for any fixed value $\eps > 0$, the convergence in \eqref{Eq12} is uniform for initial distributions $\mu$ with $\mu(Z^{(-\lambda_{0})}) \geq \eps$.
	For example, since $(0,\infty] \ni x \mapsto Z^{(-\lambda_{0})}(x)$ is strictly positive and increasing from Proposition \ref{prop:invariantFunction}, the convergence is uniform for initial distributions supported on $[\delta,\infty]$ for each $\delta > 0$.
\end{Rem}

We briefly present an application of Theorem \ref{thm:rateOfCovergence} to the exponential ergodicity of the conditional limit to avoid zero.
While the arguments are based on standard techniques (see e.g., \cite{ChampagnatVillemonaisPTRF}), we provide the proofs for completeness.
To simplify notation, we suppose $X$ is given as a coordinate process on $D$, the space of c\`adl\`ag paths from $[0,\infty)$ to $[0,\infty]$.
From Proposition \ref{prop:invariantFunction}, the process
\begin{align}
	M_{t} := \mathrm{e}^{\lambda_{0}t} \frac{Z^{(-\lambda_{0})}(X_{t})}{Z^{(-\lambda_{0})}(X_{0})}
\end{align}
is a non-negative martingale with $\bE_{x}M_{t} = 1 \ (t \geq 0)$ under $\bP_{x}$ for every $x \in (0,\infty]$, where we note that the case $x = \infty$ follows from Proposition \ref{prop:limitAtInfinity}.
This induces a change of measure, that is, for every $x \in (0,\infty]$ there exists a unique probability measure $\bQ_{x}$ on $(D,\cG_{\infty})$,where $\cG_{\infty} := \bigvee_{t > 0} \cG_{t}$ for the right-continuous induced filtration $(\cG_{t})_{t \geq 0}$, such that
\begin{align}
	\bQ_{x}[A] = \bE_{x}[M_{t},A] \quad (t \geq 0, \ A \in \cG_{t}).
\end{align}
We show that $(\bQ_{x})_{x \in (0,\infty]}$ is a limit distribution of $X$ conditioned to avoid zero.

\begin{Prop}
	For every $s \geq 0$, we have
	\begin{align}
		\lim_{t \to \infty} \bP_{x}[A \mid \tau_{0} > t] = \bQ_{x}[A] \quad (x \in (0,\infty], \ A \in \cG_{s}).
	\end{align}
\end{Prop}

\begin{proof}
	From the Markov property, \eqref{Eq11} (or \eqref{eq40}), and the bounded convergence theorem, we have for $0 < s < t$ and $A \in \cG_{s}$ 
	\begin{align}
		\bP_{x}[A \mid \tau_{0} > t] = \mathrm{e}^{\lambda_{0}s}\frac{\bE_{x}[\mathrm{e}^{\lambda_{0}(t-s)} p_{t-s}1(X_{s}),A]}{\mathrm{e}^{\lambda_{0}t} p_{t}1(x)} \xrightarrow{t \to \infty} \mathrm{e}^{\lambda_{0}s}\frac{\bE_{x}[Z^{(-\lambda_{0})}(X_{s}),A]}{Z^{(-\lambda_{0})}(x)} = \bQ_{x}[A].
	\end{align}
	The proof is complete.
\end{proof}

It is easy to see that under $\bQ_{x}$, the process $X$ is a conservative time-homogeneous Markov process on $(0,\infty]$, and its transition semigroup $(q_{t})_{t \geq 0}$ is
\begin{align}
	q_{t}f(x) = \frac{\mathrm{e}^{\lambda_{0}t}}{Z^{(-\lambda_{0})}(x)}p_{t}(Z^{(-\lambda_{0})}f)
\end{align}
for every $f \in \cB_{b}(I)$.
Define 
\begin{align}
	\mu(du) := \frac{1}{\rho}W^{(-\lambda_{0})}(0,u)Z^{(-\lambda_{0})}(u)m(du).
\end{align}
From Proposition \ref{prop:invariantFunction} (iii) the measure $\mu$ is a distribution on $(0,\infty]$. 
Applying \eqref{resolventIdentityW} and \eqref{resolventIdentityZ}, we can easily confirm that
\begin{align}
	\int_{(0,\infty)}\mu(dx)\int_{0}^{\infty}\mathrm{e}^{-rt}q_{t}f(x)dt = \frac{\mu(f)}{r} \quad (r > 0), 
\end{align}
that is, $\mu$ is an invariant distribution for $(q_{t})_{t \geq 0}$.
We show $(q_{t})_{t \geq 0}$ is exponentially ergodic.

\begin{Thm}
	Let $\delta \in (0,\lambda_{1} - \lambda_{0})$ and let $M_{\delta} \in (0,\infty)$ be the one in \eqref{Eq11}.
	Then we have for every $x \in (0,\infty]$
	\begin{align}
		\sup_{\substack{f:(0,\infty]\to \bR, \\ ||fZ^{(-\lambda_{0})}||_{\infty} \leq 1}} |q_{t}f(x) - \mu(f) | \leq \frac{M_{\delta}}{Z^{(-\lambda_{0})}(x)}\mathrm{e}^{-\delta t}.
	\end{align}
\end{Thm}

\begin{proof}
	Let $f:(0,\infty] \to \bR$ be a measurable function such that $||fZ^{(-\lambda_{0})}||_{\infty} \leq 1$.
	Noting that for every $x \in (0,\infty]$
	\begin{align}
		\mu(f) = \frac{\pi_{\lambda_{0}}(Z^{(-\lambda_{0})}f)(x)}{Z^{(-\lambda_{0})}(x)},
	\end{align}
	we have
	\begin{align}
		q_{t}f(x) - \mu(f) = \frac{\mathrm{e}^{\lambda_{0}t}p_{t}(Z^{(-\lambda_{0})}f)(x) - \pi_{\lambda_{0}}(Z^{(-\lambda_{0})}f)(x)}{Z^{(-\lambda_{0})}(x)},
	\end{align}
	and the desired result readily follows from \eqref{Eq11}.
\end{proof}

\appendix

\section{Exponential convergence to the Yaglom limit: Accessible right boundary case} \label{appendix:accessible}

In this appendix, we consider the exponential convergence to the Yaglom limit when the right boundary of the state space $I$ is accessible.
Since the proofs can be given by the obvious modifications of the arguments in Sections \ref{section:spectralTheory} and \ref{section:exponentialConvergence}, we omit the proofs and only give an outline.
Throughout, we suppose (A1)-(A3) and $\bP_{x}[\tau_{0} < \infty] > 0 \ (x > 0)$. 
We consider the following two cases: 
\begin{description}
	\item[(I)] $I = (0,\ell]$ and $\zeta = \tau_{0}$.
	\item[(II)] $I = (0,\ell)$ and $\zeta = \tau_{0} \wedge \tau_{\ell}$, where $\tau_{\ell} := \lim_{x \to \ell}\tau_{x}^{+}$.
\end{description}
Note that case (II) is treated under additional conditions, as we will see later.

We first consider case (I).
From \cite[Theorem 2.10]{noba2023analytic}, we have a representation of the resolvent density of $(p_{t})_{t \geq 0}$ killed on hitting zero:
\begin{align}
	\int_{0}^{\infty}\mathrm{e}^{-qt}p_{t}f(x)dt = \int_{(0,\ell]}r_{1}^{(q)}(x,u)f(u)m(du)
\end{align}
for
\begin{align}
	r_{1}^{(q)}(x,u) := \frac{Z^{(q)}(x,\ell)}{Z^{(q)}(0,\ell)}W^{(q)}(0,u) - W^{(q)}(x,u).
\end{align}
As suggested by this representation, this case can be obtained by just replacing $Z^{(q)}(x)$ with $Z^{(q)}(x,\ell)$ in the argument of Section \ref{section:exponentialConvergence}.
Note that $q \mapsto r^{(q)}(x,u)$ is meromorphic on $\bC$ for each $x,u \in I$.
Since the density $r^{(q)}(x,u)$ is continuous in $x$, we can easily 
confirm that $(p_{t})_{t \geq 0}$ has the Feller property, that is, $p_{t}$ maps $C_{\infty}(0,\ell]$ to itself.
We denote the infinitesimal generator by $A$.
By exactly the same argument in Section \ref{section:spectralTheory},
we can show that the spectrum $\sigma(A)$ of $A$ coincides with $\cN := \{ q \in \bC \mid Z^{(q)}(0,\ell) = 0 \}$.
Under the irreducibility, the unique existence of a quasi-stationary distribution has been shown:
\begin{Thm}[{\citet[Theorem A.6]{noba2023analytic}}]
	Suppose the following irreducibility:
	\begin{description}
		\item[(Irr1)] $\bP_{x}[\tau_{y} < \infty] > 0$ for every $x \in (0,\ell]$ and $y \in [0,\ell]$.
	\end{description}
	Let $\iota_{0} := \sup \{ \lambda \geq 0 \mid \bE_{x}[\mathrm{e}^{\lambda \tau_{0}}] < \infty \ \text{for every $x \in I$} \}$.
	Then we have $\iota_{0} \in (0,\infty)$ and 
	\begin{align}
		\iota_{0} = \min \{ \lambda \geq 0 \mid Z^{(-q)}(0,\ell) = 0 \}.
	\end{align}
	In addition, there exists a unique quasi-stationary distribution
	\begin{align}
		\mu_{\iota_{0}}(dx) = \iota_{0}W^{(-\iota_{0})}(0,x)m(dx),
	\end{align}
	and we have $\bP_{\mu_{\iota_{0}}}[\tau_{0} > t] = \mathrm{e}^{-\iota_{0}t} \ (t \geq 0)$ and $W^{(-\iota_{0})}(0,x) > 0 \ (x \in (0,\ell])$.
\end{Thm}

Assuming the strong Feller property, we have the exponential convergence to the Yaglom limit.
The following can be shown by exactly the same argument as Theorem \ref{thm:rateOfCovergence}.

\begin{Thm}
	Suppose (Irr1) and the strong Feller property, that is, $p_{t}f \in C_{\infty}(0,\ell]$ for every $f \in \cB_{b}(I)$ and $t > 0$.
	Then the spectral gap exists: $\iota_{1} > \iota_{0}$ for 
	\[
	\iota_{1} := -\sup\{ \Re q \mid q \in \sigma(A) \setminus \{ -\iota_{0} \}  \},
	\]
	and the following holds:
	for every $\delta \in (0,\iota_{1}-\iota_{0})$, there exists $M_{\delta} > 0$ and
	\begin{align}
		\sup_{f \in \cB_{b}(I), ||f||_{\infty} \leq 1}||\mathrm{e}^{\iota_{0}t}p_{t}f - \pi_{\iota_{0}}f||_{\infty} \leq M_{\delta}\mathrm{e}^{-\delta t} \quad (t > 0),
	\end{align}
	where $\pi_{\iota_{0}}$ is the projection defined by
	\begin{align}
		\pi_{\iota_{0}}f(x) := \frac{1}{\rho\iota_{0}}Z^{(-\iota_{0})}(x,\ell)\int_{I}f(y)\mu_{\iota_{0}}(dy) \label{}
	\end{align}
	for $\rho := \frac{d}{dq}Z^{(q)}(0,\ell)|_{q = -\iota_{0}} \in (0,\infty)$.
	This implies the exponential convergence to the Yaglom limit: for every distribution $\mu$ on $(0,\ell]$
	\begin{align}
		\sup_{f \in \cB_{b}(I), ||f||_{\infty} \leq 1}|\bE_{\mu}[f(X_{t}) \mid \tau_{0} > t] - \mu_{\iota_{0}}(f)| \leq
		\frac{2M_{\delta}\mathrm{e}^{-\delta t}}{\mu(Z^{(-\iota_{0})}(\cdot,\ell))} \quad (t > 0). 
	\end{align}
\end{Thm}

We next consider case (II).
To deal with this case by the same method in Section \ref{section:exponentialConvergence}, we need the scale function $W^{(q)}(x,y)$ to be defined at $y=\ell$.
For this to hold, we require the existence of an excursion measure starting from the point $\ell$.
Therefore, we assume in the following that there exists a distribution $\bP_{\ell}$, and that $(\Omega, \cF, \cF_{t}, X_{t}, \theta_{t}, \bP_{x})$ is a standard process on $(0,\ell]$ with no negative jumps satisfying (A1)-(A3).
Typically, this assumption holds when we first define the process on $(0,\infty]$ and then impose killing at the exit time of $(0,\ell)$.

In this case, the potential density is given by Theorem \ref{thm:potentialDensity} for $x=0$ and $z=\ell$:
\begin{align}
	\int_{0}^{\infty}\mathrm{e}^{-qt}p_{t}f(x)dt = \int_{(0,\ell)}r_{2}^{(q)}(x,u)f(u)m(du)
\end{align}
for
\begin{align}
	r^{(q)}_{2}(x,u) := \frac{W^{(q)}(x,\ell)}{W^{(q)}(0,\ell)}W^{(q)}(0,u) - W^{(q)}(x,u). \label{Eq37}
\end{align}
Also in this case, the density $q \mapsto r^{(q)}(x,u)$ is meromorphic in $q \in \bC$.
To ensure the Feller property, we henceforth assume the following additional assumption:
\begin{align}
	\lim_{x \to \ell}\bP_{x}[\tau_{0} < \tau_{\ell}] = 0. \label{Eq35}
\end{align}

\begin{Rem}
	Note that condition \eqref{Eq35} is equivalent to $W(\ell,\ell)=0$ from Proposition \ref{prop:exitProblem1}, and from \cite[Remark 2.5]{noba2023analytic} it is also equivalent to the fact that the point $\ell$ is regular for itself; $\bP_{\ell}[\tau_{\ell} = 0]= 1$.
\end{Rem}

We can easily check that $(p_{t})_{t \geq 0}$ has the Feller property under \eqref{Eq35}, that is, $p_{t}$ maps $C_{\infty}(0,\ell)$ to itself.
We denote the infinitesimal generator by $A$.
Similarly, we see that the spectrum $\sigma(A)$ of $A$ coincides with $\cN := \{ q \in \bC \mid W^{(q)}(0,\ell) = 0 \}$ and have the unique existence of a quasi-stationary distribution.

\begin{Thm}[{\citet[Theorem A.1]{noba2023analytic}}]
	Suppose the following irreducibility:
	\begin{description}
		\item[(Irr2)] $\bP_{x}[\tau_{y} < \infty] > 0$ for every $x \in (0,\ell)$ and $y \in [0,\ell]$.
	\end{description}
	Let $\kappa_{0} := \sup \{ \lambda \geq 0 \mid \bE_{x}[\mathrm{e}^{\lambda (\tau_{0} \wedge \tau_{\ell})}] < \infty \ \text{for every $x \in I$} \}$.
	Then we have $\kappa_{0} \in (0,\infty)$ and 
	\begin{align}
		\kappa_{0} = \min \{ \lambda \geq 0 \mid W^{(-q)}(0,\ell) = 0 \}. \label{Eq36}
	\end{align}
	In addition, there exists a unique quasi-stationary distribution
	\begin{align}
		\mu_{\kappa_{0}}(dx) = C_{\kappa_{0}}W^{(-\kappa_{0})}(0,x)m(dx),
	\end{align}
	where $C_{\kappa_{0}} > 0$ is the normalizing constant,
	and we have $\bP_{\mu_{\kappa_{0}}}[\tau_{0} \wedge \tau_{\ell} > t] = \mathrm{e}^{-\kappa_{0}t} \ (t \geq 0)$ and $W^{(-\iota_{0})}(0,x) > 0$ for $m$-almost every $x \in (0,\ell)$.
\end{Thm}

\begin{Rem}
	In \cite{noba2023analytic}, the value $\kappa_{0}$ is defined as
	\[
	\tilde{\kappa}_{0} := \sup \{\lambda \geq 0 \mid \bE_{x}[\mathrm{e}^{\lambda \tau_{0}}, \tau_{0} < \tau_{\ell}] < \infty \text{ for every $x \in I$} \}.
	\]
	We check $\kappa_{0} = \tilde{\kappa}_{0}$.
	It is obvious that $\kappa_{0} \leq \tilde{\kappa}_{0}$.
	It follows for $q \geq 0$ from \eqref{Eq37}
	\begin{align}
		\bE_{x}[\mathrm{e}^{-q(\tau_{0}\wedge \tau_{\ell})}] = 1 - q \bE_{x}[\mathrm{e}^{-q\tau_{0}}, \tau_{0} < \tau_{\ell}]\int_{(0,\ell)}W^{(q)}(0,u)m(du) + \int_{(0,\ell)}W^{(q)}(x,u)m(du).
	\end{align}
	The RHS is meromorphic on $\{ q \in \bC \mid \Re q > -\tilde{\kappa}_{0} \}$, and the $n$-th coefficient of the power series expansion around $q = 0$ can be computed as the $n$-th right-derivative of the LHS, i.e., $(-1)^{n}\bE_{x}[(\tau_{0} \wedge \tau_{\ell})^{n}]/n!$.
	Since the radius of convergence is $\tilde{\kappa}_{0}$, we see that $\bE_{x}[\mathrm{e}^{(\tilde{\kappa}_{0}-\delta)(\tau_{0}\wedge \tau_{\ell})}] < \infty$ for every $\delta > 0$.
	This shows $\kappa_{0} \geq \tilde{\kappa}_{0}$. 
\end{Rem}

The following is the exponential convergence to the Yaglom limit.
The proof can be given by almost the same arguments as Theorem \ref{thm:rateOfCovergence}.

\begin{Thm}
	Suppose (Irr2) and the strong Feller property, that is, $p_{t}f \in C_{\infty}(0,\ell)$ for every $f \in \cB_{b}(I)$ and $t > 0$.
	Then the spectral gap exists: $\kappa_{1} > \kappa_{0}$ for
	\[
	\kappa_{1} := -\sup\{ \Re q \mid q \in \sigma(A) \setminus \{ -\kappa_{0} \} \},
	\]
	and the following holds:
	for every $\delta \in (0,\kappa_{1}-\kappa_{0})$, there exists $M_{\delta} > 0$ and
	\begin{align}
		\sup_{f \in \cB_{b}(I), ||f||_{\infty} \leq 1}||\mathrm{e}^{\kappa_{0}t}p_{t}f - \pi_{\kappa_{0}}f||_{\infty} \leq M_{\delta}\mathrm{e}^{-\delta t} \quad (t > 0),
	\end{align}
	where $\pi_{\kappa_{0}}$ is the projection defined by
	\begin{align}
		\pi_{\kappa_{0}}f(x) := \frac{1}{\rho C_{\kappa_{0}}}W^{(-\kappa_{0})}(x,\ell)\int_{I}f(y)\mu_{\kappa_{0}}(dy) \label{}
	\end{align}
	for $\rho := \frac{d}{dq}W^{(q)}(0,\ell)|_{q = -\kappa_{0}} \in (0,\infty)$.
	This implies the exponential convergence to the Yaglom limit: for every distribution $\mu$ on $(0,\ell)$
	\begin{align}
		\sup_{f \in \cB_{b}(I), ||f||_{\infty} \leq 1}|\bE_{\mu}[f(X_{t}) \mid \tau_{0} \wedge\tau_{\ell} > t] - \mu_{\kappa_{0}}(f)| \leq
		\frac{2M_{\delta}\mathrm{e}^{-\delta t}}{\mu(W^{(-\kappa_{0})}(\cdot,\ell))} \quad (t > 0). 
	\end{align}
\end{Thm}

\section{Proof of Proposition \ref{prop:absenceOfNegativeJumpsEquivalence}} \label{appendix:absenceOfNegativeJumpsEquivalence}

To simplify the description, we assume that the state space $I$ is bounded, that is, $\ell_{1}, \ \ell_{2} \in \bR$.
We do not lose any generality since any case can be reduced to this situation via an order-preserving homeomorphism.

Set an optional time $\tau := \inf \{ t > 0 \mid X_{t} - X_{t-} < 0 \}$.
Take $x,y,z \in I$ with $x < y < z$.
Note that \eqref{absenceOfNegativeJumps2} is equivalent to $\bP_{x}[\tau < \infty] = 0 \ (x \in I)$.
If $\bP_{z}[\tau_{x}^{-} < \tau_{y}] > 0$, it obviously follows $\bP_{z}[\tau_{I_{<y}} < \tau_{y}] > 0$.
On $\{\tau_{I_{<y}} < \tau_{y} \}$, by the right continuity of paths, we have $X_{\tau_{I_{<y}}} < y$, which implies $X_{\tau_{I_{<y}}} - X_{\tau_{I_{<y}}-} < 0$.
Thus, we have $\bP_{z}[\tau < \infty] > 0$, and we have shown that \eqref{absenceOfNegativeJumps2} implies \eqref{absenceOfNegativeJumps1}.

Let \eqref{absenceOfNegativeJumps1} hold and suppose $\bP_{z_{0}}[\tau < \infty] > 0$ for some $z_{0} \in I$.
Set $\tau^{(\eps)} := \inf \{ t > 0 \mid X_{t} - X_{t-} < -\eps \}$ for $\eps \geq 0$.
Since $\lim_{\eps \to 0}\tau^{(\eps)} = \tau$, there exists $\delta > 0$ such that $\bP_{z_{0}}[\tau^{(\delta)} \in (0,\infty)] > 0$.
We can take an interval $J \subset I$ that is open in $I$, has length strictly less than $\delta/2$, and satisfies $\bP_{z_{0}}[A] > 0$ for 
\begin{align}
	A := \{X_{\tau^{(\delta)}-} \in J, \ \tau^{(\delta)} < \infty\}.
\end{align}
We denote the left end point of $J$ by $a$.
Specifically, we can take $J$ as either $(a,a+\delta/4)$ or $(a,\ell_{2}]$.
We denote the closure of $J$ in $I$ by $\bar{J}$.
Noting that $\tau_{J} < \tau^{(\delta)}$ on $A$, we see by the strong Markov property at $\tau_{J}$ that there exists $y_{0} \in \bar{J}$ such that $\bP_{y_{0}}[A] > 0$.
From the definition of $A$, it follows for every $y \in I_{\geq a}$
\begin{align}
	\bP_{y}[A \cap \{ \tau^{(\delta)} < \tau_{a-\delta/4}^{-}\}] \leq
	 \bP_{y}[\tau_{a-\delta/2}^{-} < \tau_{a-\delta/4}^{-} ] = 0.
\end{align}
Thus, we have $\bP_{y}[A] = \bP_{y}[A \cap \{ \tau_{a-\delta/4}^{-} < \tau^{(\delta)} \}]$.
We note that from \eqref{absenceOfNegativeJumps1} the function $I_{\geq a} \ni y \mapsto \bP_{y}[A \cap \{ \tau_{a-\delta/4}^{-} < \tau^{(\delta)} \}]$ is non-increasing.
In particular, we have
\begin{align}
	\bP_{a}[A \cap \{ \tau_{a-\delta/4}^{-} < \tau^{(\delta)} \}] \geq \bP_{y_{0}}[A] > 0.
\end{align}
Set an optional time as $\sigma := \tau_{J} \circ \theta_{\tau_{a-\delta/4}^{-}} + \tau_{a-\delta/4}^{-}$.
Noting that $\sigma \wedge \tau_{a}^{-} < \tau^{(\delta)}$ on $A \cap \{ \tau_{a-\delta/4}^{-} < \tau^{(\delta)} \}$ and $X_{\sigma} \in \bar{J}$ on $\{\sigma < \infty \}$, we have
\begin{align}
	\bP_{a}[A \cap \{ \tau_{a-\delta/4}^{-} < \tau^{(\delta)} \}]
	= &\bE_{a}[\bP_{X_{\sigma}}[A], \sigma < \tau^{(\delta)}] \\
	= &\bE_{a}[\bP_{X_{\sigma}}[A \cap \{ \tau_{a-\delta/4}^{-} < \tau^{(\delta)} \} ], \sigma < \tau^{(\delta)}] \\
	= & \bE_{a}[\bP_{X_{\sigma}}[\tau_{a}^{-} < \tau^{(\delta)}], \sigma < \tau^{(\delta)}]\bP_{a}[A \cap \{ \tau_{a-\delta/4}^{-} < \tau^{(\delta)} \} ] \\
	\leq &\bP_{a}[\tilde{\sigma} < \tau^{(\delta)}] \bP_{a}[A \cap \{ \tau_{a-\delta/4}^{-} < \tau^{(\delta)} \}],
\end{align}
where $\tilde{\sigma} := \tau_{a}^{-} \circ \theta_{\sigma} + \sigma$.
Thus, it follows $\bP_{a}[\tilde{\sigma} < \tau^{(\delta)}] = 1$.
Set recursively as $\tilde{\sigma}_{1} := \tilde{\sigma}$ and $\tilde{\sigma}_{n} := \tilde{\sigma} \circ \theta_{\tilde{\sigma}_{n-1}} + \tilde{\sigma}_{n-1} \ (n \geq 2)$.
It is easy to see that $(\tilde{\sigma}_{n+1}-\tilde{\sigma}_{n})_{n \geq 1}$ are i.i.d.\ under $\bP_{a}$ and $\bP_{a}[\tilde{\sigma_{1}} > 0] = 1$, which implies $\lim_{n \to \infty}\tilde{\sigma}_{n} = \infty$ $\bP_{a}$-almost surely.
Since $\bP_{a}[\tilde{\sigma}_{n} < \tau^{(\delta)}] = \bP_{a}[\tilde{\sigma}_{1} < \tau^{(\delta)}]^{n} = 1$,
it follows $\bP_{a}[\tau^{(\delta)} = \infty] = 1$.
This is contradiction, and the proof is complete.

\bibliographystyle{plain}

\end{document}